\documentclass[10pt,a4paper]{article}

\usepackage{geometry}
\usepackage[utf8]{inputenc}
\usepackage{amsthm, amssymb, natbib, graphicx}
\usepackage{listings}
\usepackage[ruled,vlined]{algorithm2e}
\usepackage{caption}
\usepackage{amsmath,tikz}
\usepackage{array, makecell}
\usepackage{sidecap, caption}
\usepackage{subcaption}
\usepackage{xcolor, color} 
\usepackage{comment} 
\usepackage{algorithm2e}
\usepackage[most]{tcolorbox}
\setcitestyle{numbers, comma,square}
\usepackage{bbm}
\usepackage{caption}
\usepackage{tabularray}
\usepackage{diagbox}
\UseTblrLibrary{booktabs}
\usepackage{csquotes}
\usepackage[
    pdfpagelabels,
    hypertexnames=false,  
    bookmarksopen=true,
    bookmarksopenlevel=1,  
    colorlinks=true,
    pdfstartview=FitV,
    linkcolor=blue,
    citecolor=blue,
    urlcolor=blue,
]{hyperref} 
\hypersetup{
    colorlinks=true,
    linkcolor=blue,
    filecolor=magenta,      
    urlcolor=cyan,
    pdftitle={Overleaf Example},
    pdfpagemode=FullScreen,
    }
\usepackage{cleveref}

\newtheorem{theorem}{Theorem}[section]
\newtheorem{corollary}{Corollary}[section]
\newtheorem{lemma}{Lemma}[section]
\newtheorem{example}{Example}[section]
\newtheorem{remark}{Remark}[section]

\newtheorem{definition}{Definition}[section]
\DeclareMathOperator*{\argmin}{argmin}

\DeclareMathOperator{\diag}{diag}

\newcommand{\R}{\mathbb{R}}
\def\s{{\tilde{s}}}

\def\A{{\cal A}}
\def\S{{\cal S}}
\def\Q{{\cal Q}}

\newcommand{\rr}[1]{\textcolor{red}{#1}}

\usepackage[normalem]{ulem}

\newcommand{\pp}[1]{\textcolor{purple}{#1}}
\newcommand{\bb}[1]{\textcolor{blue}{#1}}

\newcommand{\ignore}[1]{}

\title{Sufficiently Regularized Nonnegative Quartic  Polynomials are Sum-of-Squares}

\author{\thanks{wenqi.zhu@maths.ox.ac.uk, University of Oxford, UK} Wenqi Zhu,  \thanks{cartis@maths.ox.ac.uk, University of Oxford, UK} Coralia Cartis}

\author{Coralia Cartis\thanks{The order of the authors is alphabetical; the second author (Wenqi Zhu) is the primary contributor.  This work was supported by the Hong Kong Innovation and Technology Commission (InnoHK Project CIMDA).} \textsuperscript{\normalfont ,}\thanks{Mathematical Institute, Woodstock Road, University of Oxford, Oxford, UK, OX2 6GG.  {\tt coralia.cartis@maths.ox.ac.uk} } \quad and \quad
Wenqi Zhu\footnotemark[1] \textsuperscript{\normalfont,}\thanks{Mathematical Institute, Woodstock Road, University of Oxford, Oxford, UK, OX2 6GG.
{\tt wenqi.zhu@maths.ox.ac.uk}}}

\usepackage{amsmath}

\begin{document}

\maketitle

\begin{abstract}
A polynomial that is nonnegative need not be a sum of squares of polynomials. This classical gap, identified by Hilbert in 1888, lies at the heart of why the global optimization of multivariate quartic polynomials is NP-hard. Yet we show that this gap is closed when using (sufficient) regularization, which fundamentally alters the algebraic structure of the problem.
Namely, we investigate a class of quartically-regularized cubic polynomials which arise naturally in polynomial optimization and higher-order tensor methods for nonconvex problems. We show that, under mild assumptions and for sufficiently large Euclidean quartic regularization,
the shifted nonnegative polynomial becomes a sum of squares, yielding an exact semidefinite programming (SDP) formulation at the zeroth level of the Lasserre hierarchy. 
 We further derive explicit bounds on the regularization parameter that guarantee this property. Beyond this asymptotic regime, we identify structured subclasses for which SoS exactness holds for all regularization levels, including quadratic--quartic models and a class of low-rank-type cubic tensors.
In contrast, we show that separable quartic regularized polynomials---including classical tensor models proposed by Schnabel (1991)---do not, in general, induce SoS representations, even under arbitrarily large regularization. Our results reveal a sharp structural boundary between tractable and intractable regimes in polynomial optimization. In particular, they explain why Euclidean quartic regularization plays a significant role: in addition to regularising the model, it can induce exact SoS certificates and exact SDP representations. 
\end{abstract}

\section{Introduction and Overview}
\label{intro}

A foundational result in real algebraic geometry, due to Hilbert (1888), shows that nonnegativity and sum-of-squares (SoS) representations do not coincide in general \cite{hilbert1888darstellung}. Outside a few special cases, there exist nonnegative polynomials that cannot be written as sums of squares. This gap lies at the heart of polynomial optimization and semidefinite programming.

From an optimization perspective, this distinction is fundamental. If a nonnegative polynomial admits an exact SoS certificate, then its global minimum can be computed via a single semidefinite program (SDP) at the zeroth level of the Lasserre hierarchy~\cite{lasserre2001global, parrilo2001minimizing}. By contrast, in the absence of such a certificate, one typically requires a hierarchy of increasingly large SDP relaxations. Although finite convergence can occur under suitable conditions~\cite{nie2014optimality}, the required relaxation order is not known a priori and may be arbitrarily large, and the problem remains NP-hard in general. 
Classical examples illustrate this gap clearly. Nonnegative polynomials such as the Motzkin and Choi--Lam forms are not SoS \cite{motzkin1967arithmetic, choi1977extremal}, and more recently, Ahmadi~\cite[Thm.~3.3]{ahmadi2023sums} constructed an explicit nonnegative trivariate quartic polynomial that is not SoS. Furthermore, Nesterov~\cite{nesterov2003random} proved NP-hardness results for several cubic and quartic polynomial optimization problems over the Euclidean unit ball.

At the same time, many polynomial models arising in optimization are not arbitrary. In particular, structured regularized polynomials appear naturally in higher-order optimization methods and tensor-based models. These methods construct local Taylor approximations of an objective function and coupled with higher-degree regularization terms, leading to polynomial subproblems of increasing order (often referred to as AR$p$ models). 
In this paper, we focus on quartically regularized cubic polynomials of the form
\begin{equation}
\label{m3}
m_3(s)
= f_0 + g^T s + \frac{1}{2} H[s]^2
+ \frac{1}{6}\, T[s]^3
+ \frac{\sigma}{4}\, \|s\|^4,
\tag{Quartically Regularized Cubic Model}
\end{equation}
where $f_0\in\R$, $g\in\R^n$, $H\in\R^{n\times n}$, and $T\in\R^{n\times n\times n}$ are the constant, linear, quadratic, and cubic coefficients, respectively, and $\sigma>0$ is a regularization parameter; the quartic term uses the Euclidean norm $\|s\|^4=(s_1^2+\cdots+s_n^2)^2$. 
When the minimization of $m_3$ appears as an iterative subproblem in optimization algorithms for minimizing a general objective, the parameter $\sigma$ is typically chosen adaptively to control the quality of the local model and its minimizers. When $\sigma$ is sufficiently large, the model minimization is guaranteed to yield a descent step for a general objective, which is a key ingredient for ensuring global convergence of adaptive regularization methods~\cite{cartis2020concise, birgin2017worst, cartis2022evaluation, carmon2020lower}. Not only does
\eqref{m3} appear naturally in polynomial optimization and higher-order methods, but can also be studied independently as a structured class of nonconvex polynomials.



For any $\sigma>0$, the model is bounded below and admits a well-defined global minimum value $m_*$, so that the quartic polynomial $m_3 (s)-m_*$ is nonnegative for $s\in\R^n$. The central question of this paper is the following:

\begin{quote}
\emph{Under what conditions on the regularization weight $\sigma$ and problem data ($g$, $H$ and $T$) does the shifted polynomial $m_3 - m_*$  admit an exact sum-of-squares representation?}
\end{quote}




At first glance, this may appear unlikely given the 
NP-hardness of general quartic polynomial optimization. However, empirical observations and recent theoretical developments suggest that regularization may induce hidden structure. In particular, Parrilo~\cite[Sec.~5.1]{parrilo2001minimizing} observed numerically that certain quartically regularized models appear to become SoS for sufficiently large regularization. More recently, Ahmadi et al.~\cite{ahmadi2023higher} proved that sufficiently regularized \emph{locally strongly convex} quartic polynomials become {\it SoS-convex}, a related notion to SoS. However, in the general case -- namely, when \eqref{m3} is nonconvex, 
whether $m_3-m_*$ admits an exact SoS representation for sufficiently large $\sigma$ remains an open question. Similarly open is the question of whether specific values of $g$, $H$ and $T$ 
yield SoS-representable polynomials for any value of $\sigma>0$. We address both of these questions here.


The key message of our paper is that the Euclidean quartic regularization in the definition of $m_3$ creates a structured, tractable regime, when chosen appropriately. While general nonnegative quartic polynomials need not be SoS, we show that sufficiently large regularization in \eqref{m3} induces sufficient structure in the polynomial so that an exact SoS certificate can be constructed. At the same time, this behavior is not universal across quartic regularizations: certain separable quartic regularization can fail to produce SoS exactness even for arbitrarily large regularization.

The main contributions of the paper are as follows. 

\medskip
\noindent
\textbf{(i) SoS exactness under sufficient regularization.}
We show that, under mild assumptions, choosing sufficiently large Euclidean quartic regularization in \eqref{m3}
guarantees that the shifted polynomial $m_3 - m_*$ is SoS. We derive explicit bounds on the regularization parameter $\sigma$ ensuring this property. To the best of our knowledge, this is the first SoS exactness result for  generally nonconvex quartically regularized cubic polynomials. This shows that appropriate regularization can fundamentally alter SoS certifiability and enable exact semidefinite (SDP) representations.

\medskip
\noindent
\textbf{(ii) Structured polynomial classes with exact SoS certificates.}
Beyond the large-regularization regime, we identify nontrivial subclasses for which SoS exactness holds for all $\sigma > 0$. These include quadratic--quartic models and structured cubic tensors of the form $(t^T s)\|s\|^2$ (for given $t\in \R^n$ and for all $s\in \R^n$). These results extend Hilbert-type equivalences between nonnegativity and SoS to new families of multivariate polynomials.

\medskip
\noindent
\textbf{(iii) A structural boundary: Euclidean vs.\ separable regularization.}
We show that the Euclidean quartic norm plays a critical role in enabling SoS exactness. In contrast, separable quartic regularization—including classical tensor models proposed by Schnabel~\cite{schnabel1991tensor}—may fail to produce SoS representations even for arbitrarily large regularization. This reveals a sharp boundary between tractable and intractable regimes in polynomial optimization. These findings also provide insight into higher-order optimization and tensor methods: Euclidean quartic regularization not only regularizes the subproblem, but also induces a regime where exact SoS certificates and thus exact SDP formulation is attainable.

The remainder of the paper is organized as follows.
\Cref{sec: background} reviews the necessary background on SoS polynomials, SoS-convexity, the Lasserre hierarchy, and related work in polynomial optimization.
\Cref{sec: main theory} contains the main theorem, proving that sufficiently strong Euclidean quartic regularization forces SoS exactness for the quartically regularized cubic model.
\Cref{sec: Special Classes of SoS Quartically Regularized Polynomial} studies several structured subclasses, including cases where SoS exactness holds for all $\sigma>0$, and contrasts them with separable quartic regularization models where SoS exactness may fail.
\Cref{sec examples} presents examples illustrating why the geometry of the regularization term is decisive, and explains how our results differ from classical NP-hardness examples.
Finally, \Cref{sec: conclusion} concludes with a discussion of the implications for higher-order optimization, tensor methods, and future directions.

\subsection{Background: Nonnegativity, Sum of Squares, and SDP}
\label{sec: background}

Let $\mathcal{P}_n^{\deg(h)}$ denote the class of real polynomials $h:\mathbb{R}^n \to \mathbb{R}$ of degree $\deg(h)$. 
We denote by $h^* := \min_{s \in \mathbb{R}^n} h(s)$ the global minimum when it exists. A polynomial $h(s)$ is said to be \emph{nonnegative} if $h(s) \ge 0$ for all $s \in \mathbb{R}^n$.

\begin{definition}[SoS polynomial]
A polynomial $h(s)$ is a sum of squares (SoS) if there exist polynomials $h_1,\dots,h_r$ such that
$
h(s) = \sum_{j=1}^r h_j(s)^2.
$
We say that $h$ is \emph{SoS-exact at level 0} if $h(s) - h^*$ is SoS.
\end{definition}

Hilbert \cite{hilbert1888darstellung} gives that nonnegativity and SoS coincide only in limited cases (e.g., univariate polynomials or quadratic forms), while in general they differ. 
A key advantage of SoS representations is their connection to semidefinite programming (SDP). 
It is well known that a polynomial is a sum of squares if and only if it admits a positive semidefinite Gram matrix representation. This allows global minimization of SoS polynomials via a single SDP.

\begin{theorem} (Parrilo \cite{parrilo2000structured})
    Let $h \in \mathcal{P}_{n}^{2d}$ and
    $
    \phi_d(s) $
    be the vector of all monomials of degree at most $d$. A polynomial $h(s)$ is SoS if and only if there exists a symmetric matrix $\Q$ such that (i) $h(s) =\phi_d(s)^T\Q\phi_d(s)$ for all $s \in \R^n $, (ii) $\Q \succeq 0$. Note that $\Q$ has size $\binom{n+2d}{2d}$.
    \label{thm SoS alternative}
\end{theorem}

If $h(s) - h^*$ is SoS, then its global minimum can be computed by solving a single semidefinite program via its Gram matrix representation~\cite{lasserre2000optimisation, parrilo2001minimizing}. 
Moreover, under mild conditions (e.g., uniqueness), the global minimizer can be recovered from the solution of the corresponding moment SDP~\cite{nie2013certifying}. When $h(s) - h^*$ is not SoS, one can instead use the Lasserre hierarchy of SDP relaxations, which converges asymptotically and, under suitable conditions, in finitely many steps~\cite{lasserre2001global, nie2014optimality}. However, the required relaxation order is not known a priori and may be arbitrarily large, so the problem remains NP-hard in general~\cite{scheiderer2006distinguished}. We will also use the notion of SoS-convexity. A polynomial $h$ is called SoS-convex if its Hessian is an SoS matrix. In this case, nonnegativity together with first-order optimality implies an SoS representation~\cite{helton2010semidefinite}.  Recently, Ahmadi et al.~\cite{ahmadi2023higher} showed that sufficiently regularized locally strongly convex quartic polynomials become SoS-convex\footnote{More details in SoS convex are given in \Cref{def sos convex}.}, with an algorithmic framework developed in \cite{zhu2024global}, highlighting the tractability of the convex setting.
Helton and Nie~\cite{helton2010semidefinite} showed that convex semialgebraic sets satisfying suitable curvature or SoS-type conditions admit SDP representations. 
However, when the model is nonconvex, the SoS behavior of quartically regularized polynomials remains largely unexplored. 

\paragraph{Related work on high-order Methods}
There is a substantial body of work on higher-order optimization methods, in both convex and nonconvex settings. The adaptive $p$th order regularization framework (AR$p$) has been extensively studied in \cite{cartis2020concise, birgin2017worst, cartis2022evaluation, carmon2020lower}. In particular, it is known that sufficiently large regularization ensures descent and guarantees optimal evaluation complexity. In the nonconvex setting, AR$p$ methods require at most 
$O\left(\epsilon_g^{-\frac{p+1}{p}}\right)$
evaluations to achieve an $\epsilon_g$ approximate first order critical point and is optimal for this class of problems \cite{birgin2017worst, cartis2020concise, cartis2022evaluation, carmon2020lower}. Improved complexity bounds are available in the convex case \cite{nesterov2021implementable}, and higher-order methods also possess faster local convergence rates \cite{doikov2022local, welzel2025local}. These results highlight the practical and theoretical advantages of using higher-order models \cite{cartis2024efficient}. On the iterative algorithmic side, Nesterov and collaborators developed efficient solvers for convex higher-order models based on Bregman techniques and higher-order proximal methods \cite{nesterov2021implementable, nesterov2020inexact, nesterov2021inexact, nesterov2021superfast, nesterov2022quartic, nesterov2006cubic}. In the nonconvex setting, recent approaches exploit structured quartic regularization, including quadratic--quartic regularization (QQR) \cite{zhu2022quartic}, cubic--quartic regularization (CQR) \cite{zhu2023cubic}, and the diagonally tensor method (DTM) \cite{zhu2025global}. These methods leverage higher-order structure to design tractable subproblem solvers.

From a polynomial optimization perspective, Ahmadi and Zhang~\cite{ahmadi2022complexity} established polynomial-time checkable conditions for local minimality of cubic polynomials. Building on this, Silina et al.~\cite{silina2022unregularized} proposed an unregularized third-order Newton method that computes {\it local} minimizers of $m_3$ with $\sigma=0$, using SDPs, with a follow-up global algorithmic framework in \cite{cai2026globally}.  A broader overview of higher-order methods can be found in~\cite[Ch.~4]{cartis2022evaluation}.

\section{Regularization-Induced Sum-of-Squares Exactness}
\label{sec: main theory}

{
{

While general nonnegative quartic polynomials need not admit sum-of-squares (SoS) representations, we show that sufficiently large quartic regularization fundamentally alters this picture, inducing sufficient change in structure to 
make $m_3-m_*$ SoS representable.

We now explain the methodology underlying our main result. 
We first derive a constructive SoS representation based on global optimality conditions for \eqref{m3}. We then quantify how the size of the minimizer scales with the regularization parameter. Finally, we show that the regularization term eventually dominates all remaining terms, ensuring that the resulting representation of $m_3-m_*$ is indeed SoS.

The two central results of this section are \Cref{thm: certify SoS Quartically Regularized Polynomial} and \Cref{thm B psd}.  
The first gives a constructive SoS certificate conditional on $B(s^*)\succeq 0$, while the second shows that this condition is automatically satisfied once $\sigma$ is sufficiently large. 
We also discuss two important regimes separately. 
In the locally strongly convex case, sufficiently large regularization leads to global convexification and SoS-convexity; see \Cref{sec: Locally convex m3}. 
In the locally nonconvex case, the model remains nonconvex near the origin, but we show that sufficiently large regularization still drives all relevant stationary points into a region where global optimality and SoS exactness can be certified; see \Cref{sec: Locally nonconvex m3 necc suff}. 
These two regimes together clarify how regularization, convexification, SoS certificates, and global optimality interact in quartically regularized cubic models.
}}

\subsection{A Constructive SoS Representation via Global Optimality Conditions}
\label{sec he SoS expression of m3 and its conditions}

{{
We begin by deriving a constructive sufficient condition under which the shifted polynomial
$
q(v):=m_3(s^*+v)-m_3(s^*)
$
is a sum of squares. 
The strategy is to rewrite $q(v)$ so that its quartic and cubic terms can be reorganized into explicit square expressions plus one residual quadratic form in $v$. 
The latter will be absorbed by a symmetric matrix $B(s^*)$. 
If this matrix is positive semidefinite, then the whole polynomial becomes SoS.

The first step is an exact expansion of the shifted model around $s^*$. 
Using the global optimality conditions of $m_3$ in \cite{zhu2025global}, we obtain the following representation. An equivalent integral representation of the shifted model and its connection 
to global optimality conditions are provided in Appendix~\ref{appendix integral representation}. 
}}

\begin{lemma}(Adapted from \cite[Lemma~2.1]{zhu2025global})
Let $q(v) := m_3(s^*+v) - m_3(s^*)$ where $m_3$ is defined in \eqref{m3}.  The first- and second-order derivatives of $m_3$ are
\[
\nabla m_3(s)=g+Hs+\frac12\,T[s]^2+\sigma\|s\|^2 s,
\qquad
\nabla^2 m_3(s)=H+T[s]+\sigma\bigl(2ss^T+\|s\|^2 I_n\bigr).
\]
where $I_n \in \R^{n \times n}$ is the identity matrix.
Then the fourth-order Taylor expansion of $q(v)$ at the point $s^*$ is exact and yields, for any vector $v \in \mathbb{R}^n$,
\begin{align}
q(v)
= \nabla m_3(s^*)^T v
   + \frac{1}{2}\,\nabla^2 m_3(s^*)[v]^2
   + \frac{1}{6}\,T[v]^3
   + \frac{\sigma}{4}\|v\|^4
   + \sigma\, {s^*}^T v\, \|v\|^2.
\label{expression m3 using m_*}
\end{align}
\label{remark q(v) expression}
\end{lemma}

Next, we present a technical lemma that rewrites the homogeneous cubic term $T[v]^3$, as a sum of linear polynomials multiplied by quadratic polynomials.

\begin{definition}
For quartic polynomials ($d=2$), $\phi_2$ is a list of monomial basis with highest degree $2$, 
$$
\phi_2(v) := [1, v_1, \dotsc, v_n, v_1^2, \dotsc, v_n^2, v_1v_2, \dotsc, v_1v_n, v_2v_3, \dotsc, v_2v_n, \dotsc, v_{n-1}v_n]^T \in \R^{ \frac{1}{2}(n+2)(n+1)}.
$$
Let $\omega \in \R^{\frac{1}{2}(n+2)(n+1)}$, we simplify notations $ \omega:= [1,v, u, z]^T $ where $v = [v_1, \dotsc, v_n]^T$, $u = [v_1^2, \dotsc, v_n^2]$ and $z = \{z_{ij}\}_{1 \le i < j \le n}= [v_1v_2, \dotsc, v_1v_n, v_2v_3, \dotsc, v_2v_n, \dotsc, v_{n-1}v_n]$. 
\label{phi2}
\end{definition} 

\begin{lemma}
\label{lemma cubic form}
    Any homogeneous cubic polynomial,  $c_3: \R^n \rightarrow \R$ with $\deg(c_3) = 3$ can be expressed as
    $$
    c_3(v) =  \bigg(\sum_{i=1}^n u_i \bigg) \big( t^Tv\big) + \bigg[\sum_{1 \le i<j \le n} z_{ij}  \big({v^T b^{(ij)}}\big) \bigg], 
    $$ 
    where $u_i$ and $z_{ij}$ are quadratic terms, $t$ and $b^{(ij)}$ are vectors in $\R^n$ based on the coefficients of $  c_3(v) $. 
\end{lemma}

\begin{proof}
We provide the full derivation in Appendix~\ref{appendix proof lemma cubic}. 
\end{proof}

The next theorem formalizes the constructive part of the argument. 
It shows that once the residual matrix $B(s^*)$ is positive semidefinite, the shifted polynomial $q(v)$ can be written explicitly as a sum of squares. 
In other words, the difficulty of proving SoS exactness is reduced to a matrix positivity condition. 

\begin{theorem}
Assume $\sigma>\nu>0$. Let $s^*$ satisfy $\nabla m_3(s^*)=0$ and
\small
\begin{align}
\label{eq:B-sstar}
B(s^*) 
= H + T[s^*]    + \sigma\|s^*\|^2 I_n & {- \nu \big(\,\|s^*\|^2 I_n - s^* s^{*T}\big)}  - \frac{2}{\sigma-\nu}\, t t^T  
      - 2 (s^* t^T + t s^{*T}) \notag \\
&\quad - \sum_{1 \le i < j \le n} 
   \left(
      \tilde{s}^{(ij)} b^{(ij)T}
      + b^{(ij)} \tilde{s}^{(ij)T}
      + \frac{1}{\nu} b^{(ij)} b^{(ij)T}
   \right)
   \succeq 0.
\end{align}
\normalsize
where $\tilde{s}^{(ij)} = [0, \dotsc s^*_j, \dotsc, s^*_i, \dotsc 0] \in \R^n$ has $s^*_j$ on $i$th entry and $s^*_i$ on $j$th entry and all other entries are zero; $ b^{(ij)}$ and $t$ depend on tensor coefficients and are defined in \eqref{Tv expression}. 
Then, \(q(v) := m_3(s^* + v) - m_3(s^*)\)  is SoS in \(v\), 
and consequently, \(s^*\) is a global minimizer of $m_3$. 
\label{thm: certify SoS Quartically Regularized Polynomial}
\end{theorem}

\begin{proof}[Proof sketch] 
We provide only the main idea here; the full algebraic derivation is given in Appendix \ref{appendix proof thm 2.1}. 
The proof is constructive. We rewrite the shifted quartic polynomial
$
q(v):=m_3(s^*+v)-m_3(s^*)
$
as a sum of explicit square terms plus one residual quadratic form, and then show that the latter is nonnegative whenever $B(s^*)\succeq 0$. 
Since $s^*$ is a critical point, $\nabla m_3(s^*)=0$, and by \Cref{remark q(v) expression} the fourth-order Taylor expansion is exact:
\[
q(v)
=
\frac{1}{2}\nabla^2 m_3(s^*)[v]^2
+\frac{1}{6}T[v]^3
+\sigma (s^*)^\top v\,\|v\|^2
+\frac{\sigma}{4}\|v\|^4.
\]
The challenge is therefore to control the cubic term $T[v]^3$ and the mixed quartic term $\sigma (s^*)^\top v\,\|v\|^2$ in a way compatible with an SoS decomposition.

To do so, we first rewrite the cubic term in the form given by \Cref{lemma cubic form}, introducing the quadratic monomials
$
u_i=v_i^2,
\qquad
z_{ij}=v_iv_j.
$
This representation expresses the tensor contribution as linear forms multiplied by quadratic monomials, which aligns it with the quartic structure induced by the regularization. 
We then introduce an auxiliary parameter $\nu\in(0,\sigma)$ and complete squares in three groups: terms involving the diagonal monomials $u_i$, terms involving the cross monomials $z_{ij}$, and mixed terms involving $s^*$. 
This yields a decomposition of the form
\[
2q(v)=\mathcal{SI}_1+\mathcal{SI}_2+\mathcal{SI}_3+B(s^*)[v]^2,
\]
where $\mathcal{SI}_1,\mathcal{SI}_2,\mathcal{SI}_3$ are explicit sums of squares, and $B(s^*)$ is the residual symmetric matrix in \eqref{eq:B-sstar}. 
Therefore, if $B(s^*)\succeq 0$, then the residual quadratic form $B(s^*)[v]^2$ is itself a sum of squares, and hence $q(v)$ is SoS. Since $q(v)\ge 0$ for all $v$, this also certifies that $s^*$ is a global minimizer of $m_3$.
\end{proof}

 The conditions in \Cref{thm: certify SoS Quartically Regularized Polynomial} imply that both first- and second-order local optimality conditions of $m_3$ are satisfied at $s^*$; namely, $
\nabla m_3(s^*) = 0
$
and  
$
\nabla^2 m_3(s^*) \succeq 0$, respectively. The proof is provided in Appendix \ref{appendix proof sos local min}.

\begin{remark} 
The SoS certificate condition established in \Cref{thm: certify SoS Quartically Regularized Polynomial} is closely related to global optimality conditions. 
For example, when \(T = 0\), by \cite{cartis2007adaptive, cartis2023second},  a point \(s^*\) is a global minimizer of \(m_3\) \emph{if and only if}
\begin{equation}
\nabla m_3(s^*) = 0,
\qquad 
B(s^*) = H + \sigma \|s^*\|^{2} I_n \succeq 0,
\end{equation}
which coincides with \eqref{eq:B-sstar}.  Yet in the case of $ T\neq 0$, the condition in \eqref{eq:B-sstar} does not, in general, coincide with the global optimality conditions for 
 $m_3$ derived in \cite[Sec.~3]{zhu2025global}.) 
\end{remark}

The next corollary gives a more standard SoS formulation of our result, showing that we are able to identify a $\Q$ matrix as needed in \Cref{thm SoS alternative}. 

\begin{corollary}
\label{Q corollary}
    ({$\Q$ matrix formulation})  
Let $\mathbf{1}^T = [1, \dotsc, 1] \in \mathbb{R}^n$. We construct a symmetric matrix of dimension $\tbinom{n}{2} \times \tbinom{n}{2}$, such that
 \small
\begin{eqnarray}
\label{Q general form}
 \Q =
\begin{bmatrix}
0 & 0 & 0 & 0 \\[3pt]
0 & \nabla^2 m_3(s^*) & D(s^*) & C(s^*) \\[3pt]
0 & D(s^*)^T & U & 0 \\[3pt]
0 & C(s^*)^T & 0 & \nu I_{N}
\end{bmatrix}
\end{eqnarray}
\normalsize
\begin{itemize} \setlength{\itemindent}{0pt}
    \item The $u$-$u$ block has size $n \times n$ and $  U =   \Big(\frac{\sigma}{2} - \frac{\nu}{2}\Big)\mathbf{1}\mathbf{1}^T
+ \frac{\nu}{2} I_n$  which has   $\frac{\sigma}{2}$ on the diagonal and   $\frac{\sigma - \nu}{2}$ off-diagonal.
\item The   $z$-$z$ block has size $N \times N$ is an identity matrix $\nu  I_N$.
\item  The $v$-$u$ block has size $n \times n$, $
  D(s^*)  = \big[( \sigma - \nu)  s^* +  t \big]\mathbf{1}^T+ \nu \diag(s^*)$. 
  \item The $v$-$z$ block has size $n \times N$ with the $ij$-th column as $\nu \tilde{s}^{(ij)} + b^{(ij)}  
$. 
\end{itemize}
If the conditions in \Cref{thm: certify SoS Quartically Regularized Polynomial} are satisfied, then $\Q \succeq 0$.
\end{corollary}

\begin{proof}
We provide the full derivation in Appendix~\ref{appendix proof Q corollary}. 
\end{proof}

\begin{example}
    We write the full matrix for the case of $n=2$:
\small
$$
2q(v)  = 
 \begin{bmatrix}
  1 \\ v_1  \\ v_2 \\ u_1   \\ u_2   \\ z_{12}
 \end{bmatrix} 
 \begin{bmatrix} 
 0 &  0  &  0 & 0 & 0  & 0
\\0 &  \nabla^2 m_3(s^*) &  & \sigma s_1 + t_{111} & (\sigma - \nu)  s_1 + t_{111} &  \nu  s_2  + t^*_{112} 
\\0 &   &  & (\sigma - \nu) s_2 + t_{222}  &  \sigma s_2 + t_{222} &  \nu  s_1  + t^*_{122} 
\\ 0 &  \ast  &  \ast &   \frac{\sigma}{2} &  \frac{\sigma}{2} - \frac{\nu}{2} & 0
\\ 0 &  \ast &  \ast   &  \frac{\sigma}{2} - \frac{\nu}{2}  &  \frac{\sigma}{2} & 0
\\ 0 &    \ast &   \ast &   0  &   0 & \nu
 \end{bmatrix}
\begin{bmatrix}
  1 \\ v_1  \\ v_2 \\ u_1   \\ u_2   \\ z_{12}
\end{bmatrix}.
$$ 
\normalsize
The symbol “$*$” denotes symmetric entries in the corresponding positions. By examining the quadratic form \(\Q[\tilde{\omega}]^2\), 
with \(\tilde{\omega} \in \mathbb{R}^{\frac{1}{2}(n+2)(n+1)}\), 
we find that \(\Q[\tilde{\omega}]^2 \ge 0\) holds 
whenever \(\nabla m_3(s^*) = 0\) and \(B(s^*) \succeq 0\).  
It follows that \(\Q \succeq 0\), and consequently 
the corresponding polynomial \(q(v)\)  is SoS. Lastly, we note that the expression of $\Q$ is not unique up to a factor of  $\nu$, $\delta$, and $\beta$ (see \Cref{other form of Q n=2}). 
\end{example}

\subsection{Scaling of Global Minimizers under Large Regularization}
\label{sec step size bound for s star}

{

Theorem \ref{thm: certify SoS Quartically Regularized Polynomial} reduces the SoS question to the matrix condition $B(s^*)\succeq 0$. 
To verify this condition, we need quantitative control on the size of the global minimizer $s^*$. 
The key point is that the regularization contributes the term $\sigma \|s^*\|^2 I_n$ to $B(s^*)$, so its strength depends not only on $\sigma$ itself but also on the scale of $s^*$.

The next theorem shows that, under mild assumptions, the minimizer obeys the scaling law
\[
\|s^*\|=O(\sigma^{-1/3}).
\]
This mechanism forces $B(s^*)$ to become positive semidefinite as $\sigma$ sufficiently large. The vector norm, the matrix norm, and the tensor norm are defined and given in Appendix~\ref{appendix tensor 2 norm and F norm}.
}


\begin{theorem} \textbf{(Bound for $s^*$)}
Assume that $\|g\| \neq 0$, let $s^*$ satisfies $\nabla m_3 (s^*) = 0$ and $ m_3 (s^*) \le m_3 (0) $, and $\sigma$ satisfies
\begin{eqnarray}
    \sigma \ge  \max \bigg\{\frac{9 \|H\|^3 }{\|g\|^2}, \bigg( \frac{3 \|T\|^3}{ 8 \|g\|}\bigg)^{1/2}  \bigg\}, 
    \label{bound for sigma thm}
\end{eqnarray}
then, 
\begin{eqnarray}
\bigg(\frac{\|g\|}{3 \sigma} \bigg)^{1/3} \le \|s^*\| \le 2 \bigg(\frac{\|g\|}{ \sigma} \bigg)^{1/3}  .
    \label{bound for s thm}
\end{eqnarray}
\label{thm bound for $s^*$}
\end{theorem}

\begin{proof}
\noindent
\textbf{To prove the upper bound for $\|s^*\|$}, we rewrite $m_3$ at $s^*$ as
\begin{eqnarray*}
m_3(s^*) - m_3(0) &=&  \bigg( g^T s^*  + \frac{\sigma}{8} \|s^*\|^4\bigg)  + \bigg( \frac{1}{2}H[s^*]^2 + \frac{\sigma}{16} \|s^*\|^4 \bigg) + \bigg(\frac{1}{6}T[s^*]^3+\frac{\sigma}{16} \|s^*\|^4 \bigg) 
\\ &\ge&  \bigg(- \|g\| + \frac{\sigma}{8} \|s^*\|^3\bigg)\|s^*\|  + \bigg(-\frac{\lambda_*}{2} + \frac{\sigma}{16} \|s^*\|^2 \bigg)\|s^*\|^2 + \bigg(-\frac{\|T\|}{6}+\frac{\sigma}{16} \|s^*\| \bigg) \|s^*\|^3
\end{eqnarray*}
where the last inequality comes from $g^T s  \le  \|g\| \|s\|$, $ H[s^*]^2  \ge -\lambda_* \|s^*\|$ and  $ T[s^*]^3  \ge -\|T\|\|s^*\|^3$. 
Assume that \eqref{bound for sigma thm} holds and assume for contradiction the global minimum of $m_3(s)$ happens at $\|s^*\|> r_c: = \big(\frac{8\|g\|}{ \sigma} \big)^{1/3} $.  Then, 
\begin{eqnarray}
m_3(s^*) - m_3(0)>  \underbrace{\bigg[- \|g\| + \frac{\sigma}{8} r_c^3\bigg]}_{= 0}\|s^*\|  + \underbrace{\bigg[-\frac{\lambda_*}{2} + \frac{\sigma}{16} r_c^2 \bigg]}_{\ge 0}\|s^*\|^2 + \underbrace{\bigg[-\frac{\|T\|}{6}+\frac{\sigma}{16}  r_c \bigg]}_{\ge 0} \|s^*\|^3 \ge  0.
\label{eq contr}
\end{eqnarray}
However, we have $m_3(s^*) \le m_3(0)$, we arrives at a contradiction. 
Note that in \eqref{eq contr}, the first term is zero using $r_c=   \big(\frac{8\|g\|}{ \sigma} \big)^{1/3}$. The second and third term is positive using $\sigma \ge \frac{9 \|H\|^3 }{\|g\|^2}$ and $\sigma \ge \big( \frac{3 \|T\|^3}{ 8 \|g\|}\big)^{1/2}$ from condition \eqref{bound for sigma thm}, 
\begin{eqnarray*}
 \frac{\sigma}{16} r_c^2  &=&   \bigg(\frac{1}{4} \|g\|^{2/3} \bigg) \sigma^{1/3} \underset{\eqref{bound for sigma thm}}{\ge}  \bigg(\frac{1}{4} \|g\|^{2/3} \bigg)  9^{1/3} \|H\| \|g\|^{-2/3} > \frac{\|H\|}{2} > \frac{\lambda_*}{2},
\\
 \frac{\sigma}{16}  r_c  &=&   \bigg(\frac{1}{8} \|g\|^{1/3} \bigg) \sigma^{2/3} \underset{\eqref{bound for sigma thm}}{\ge}  \bigg(\frac{1}{8} \|g\|^{1/3}  \bigg) \bigg(\frac{8}{3}\bigg)^{1/3} \|T\| \|g\|^{-1/3} > \frac{\|T\|}{6}.
\end{eqnarray*}

\noindent
\textbf{To prove the lower bound for $\|s^*\|$}, 
\begin{eqnarray}
\big\|g\big\| &=& \big\|\nabla m_3(0)\big\| = \big\|\nabla m_3(s^*) - \nabla m_3(0)\big\| \le  \bigg\| H[s^*]+ \frac{1}{2} T[s^*]^2  + \sigma \big\|s^*\big\|^2s^*\bigg\| \notag
\\ &\le&  \| H\|\|s^*\|+ \frac{1}{2}  \|T\|\|s^*\|^2  + \sigma \|s^*\|^3. 
\label{lip on m_3}
\end{eqnarray}
Therefore, we have 
$$
\frac{\|g\|}{3}  \le  \| H \|  \| s^*\big\|, 
\quad \text{or}  \quad 
\frac{\|g\|}{3}  \le   \frac{1}{2} \|T\|\|s^*\|^2,   
\quad \text{or}  \quad 
\frac{\|g\|}{3}  \le  \sigma \|s^*\|^3.  
$$
This gives 
\begin{eqnarray}
\|s^*\| \ge \min \bigg\{ \frac{\|g\|}{3 \| H \| },  \bigg(\frac{2\|g\|}{3  \|T\|}\bigg)^{1/2}, \bigg(\frac{\|g\|}{3 \sigma} \bigg)^{1/3} \bigg\}. 
\label{lower bound for sstar}
\end{eqnarray}
From \eqref{bound for sigma thm}, we have 
\begin{eqnarray*}
\sigma \ge \frac{9 \|H\|^3 }{\|g\|^2},  \qquad &\Rightarrow& \qquad  \bigg(\frac{\|g\|}{3 \sigma} \bigg)^{1/3}   \le \frac{\|g\|}{3 \| H \| }
\\
\sigma \ge \bigg( \frac{3 \|T\|^3}{ 8 \|g\|}\bigg)^{1/2}  \qquad &\Rightarrow&  \qquad  \bigg(\frac{\|g\|}{3 \sigma} \bigg)^{1/3}   \le \bigg(\frac{2\|g\|}{3  \|T\|}\bigg)^{1/2}. 
\end{eqnarray*}
Therefore, under \eqref{bound for sigma thm},  \eqref{lower bound for sstar} simplifies to $\|s^*\| \ge \big(\frac{\|g\|}{3 \sigma} \big)^{1/3} .$
\end{proof}

\begin{remark} (Global Minimum Attained)
{For any $\sigma > 0$, the model $m_3(s)$ is bounded below and coercive and
$
\|s\| \to \infty $ gives $ m_3(s) \to +\infty .
$
The argument used to derive an upper bound on $\|s^*\|$ implies that
$m_3(s) \le m_3(0)$ can only occur when $\|s\| \le r_c$ for some positive $r_c$. Consequently, the constrained
minimization problem
$
s_r^* := \arg\min_{\|s\| \le r_c} m_3(s)
$
admits a global minimizer (not just a finite global minimum). Since any global minimizer $s^*$ of $m_3$ must satisfy
$m_3(s^*) \le m_3(0)$, we conclude that $\|s^*\|\leq r_c$,
and hence a global minimizer of
$m_3$ is attained for any $\sigma > 0$.}
    \label{remark global min attained}
\end{remark}


\subsection{Importance of the Regularization Term}
\label{sec geneeral B psd}

{

We now combine the two ingredients developed above: the constructive SoS criterion from \Cref{thm: certify SoS Quartically Regularized Polynomial} and the scaling estimate for $s^*$ from \Cref{thm bound for $s^*$}. 
The goal is to show that the residual matrix $B(s^*)$ becomes positive semidefinite once the regularization is sufficiently large.

The key idea is that in the expression of $B(s^*)$, the dominant positive contribution comes from the term $\sigma \|s^*\|^2 I_n. $
By \Cref{thm bound for $s^*$}, this term grows as $\sigma^{1/3}$ as $\sigma\to\infty$ which acts as a uniform positive shift in all directions. 
All remaining terms in $B(s^*)$ are lower-order contributions depending on $H$, $T$, $t$, $b^{(ij)}$, and $s^*$. 
We show below that, for sufficiently large $\sigma$, the regularization term dominates these indefinite components, so that
\[
B(s^*)\succeq 0.
\]
This yields the main conclusion of the section: sufficiently strong quartic regularization forces the shifted model to become SoS. Moreover, the argument provides an explicit lower bound on $\sigma$ in terms of the coefficients of $m_3$.
}


\begin{theorem} \textbf{($B(s^*)$ is PSD for large $\sigma$)}
Assume that $\|g\| \neq 0$, let $s^*$ satisfies $\nabla m_3 (s^*) = 0$ and $ m_3 (s^*) \le m_3 (0) $, and $\sigma$ satisfies
\begin{eqnarray}
   {{ \sigma \ge  \max \bigg\{\frac{9 \|H\|^3 }{\|g\|^2},  \bigg( \frac{3 \|T\|^3}{ 8 \|g\|} \bigg)^{1/2},   \bigg( \frac{3\cdot 32^3 \|T\|^3}{  \|g\|} \bigg)^{1/2},  \frac{ 3^2\cdot 6^3\lambda_*^3 }{\|g\|^2}\bigg\}= \max \bigg\{\frac{9\cdot 6^3 \|H\|^3 }{\|g\|^2},  \bigg( \frac{3\cdot 32^3 \|T\|^3}{  \|g\|} \bigg)^{1/2} \bigg\} }}
    \label{bound for sigma B psd_new}
\end{eqnarray}
where the norms and  $\lambda_*$ are defined in  \eqref{def of norm 1} -- \eqref{def of norm 2}. Then $B(s^*)$ defined in \eqref{eq:B-sstar} with $\nu = \frac{\sigma}{2}$  is positive semi-definite. Moreover, then, $m_3(s^*+v)-m_3(s^*)$ is SoS and consequently, \(s^*\) is a global minimizer. 
\label{thm B psd}
\end{theorem}

\begin{proof}
When  $\nu = \frac{\sigma}{2}$ in \eqref{eq:B-sstar}, we have
\small
\begin{eqnarray*}
B(s^*)= H + T[s^*] + \frac{\sigma}{2} \|s^*\|^2 I_n + \frac{\sigma}{2}  s^*{s^*}^T 
      - \frac{4 t t^T}{\sigma} 
      - 2 \big(s^* t^T + t {s^*}^T\big)
      -\sum_{1 \le i < j \le n} 
    \Big[
      \tilde{s}^{(ij)} b^{(ij)T}  + b^{(ij)} \tilde{s}^{(ij)T}
      + \frac{2b^{(ij)} b^{(ij)T} }{\sigma}
    \Big] 
\end{eqnarray*}
\normalsize 
We give some lower bounds for the eigenvalues of the terms in $B(s^*)$, using the inequalities
\begin{eqnarray}    
 {-2 \|t\|\cdot\|s^*\| I_n \preceq} \,\,s^* t^T + t {s^*}^T &\preceq & 2 \|t\|\cdot\|s^*\| I_n,
 \qquad   0  \preceq  t t^T \preceq  \|t\|^2 I_n,
 \label{rank one matrix bound 1}
\\
  0  \preceq \sum_{1 \le i < j \le n} b^{(ij)} b^{(ij)T} &\preceq & \sum_{1 \le i < j \le n} \|b^{(ij)}\|^2 I_n  \preceq  \|T\|_F^2 I_n, \quad 0  \preceq    s^*{s^*}^T,  \notag
\\
  {-2 \|T\|_F \|s^*\| I_n \preceq}   \sum_{1 \le i < j \le n}  \Big[  \tilde{s}^{(ij)} b^{(ij)T}  + b^{(ij)} \tilde{s}^{(ij)T}  \Big]  &\preceq &  2 \sum_{1 \le i < j \le n}  \| \tilde{s}^{(ij)}  \| \cdot \| b^{(ij)T}   \| I_n  
 \preceq 2 \|T\|_F \|s^*\| I_n. \notag
  \label{rank one matrix bound 2}
\end{eqnarray}
Using the upper bounds in \eqref{rank one matrix bound 1},
{{and definition of norms in \eqref{def of norm 1} -- \eqref{def of norm 2}, we have
\begin{eqnarray}
   \lambda_{\min} \big[B(s^*) \big] &\ge& -\lambda_* - \|T\| \|s^*\| +  \frac{\sigma}{2}\|s^*\|^2 - \frac{4}{\sigma} \|t\|^2  - 4 \|t\|\|s^*\| - 2 \|T\|_F \|s^*\| - \frac{2}{\sigma}  \|T\|_F^2,
   \label{min evalue of B}\\
   &\ge& \frac{\sigma}{2}\|s^*\|^2-7 \|T\|_F\|s^*\|-\frac{6}{\sigma}\|T\|^2_F-\lambda_*,\label{min evalue of B_new}
\end{eqnarray}
where in the second inequality, we used 
 $\|T\|_F \ge \|T\|$ and $\|T\|_F \ge \|t\|$,  and re-arrangements. Thus, $\lambda_{\min} \big[B(s^*) \big]\geq 0$
 if the right-hand side of \eqref{min evalue of B_new} is non-negative, which in turn is equivalent to $\|s_*\|$ being greater or equal to the positive root of  the right-hand side of \eqref{min evalue of B_new} (as a function of $\|s_*\|$), namely, $\|s^*\|\geq \frac{7}{\sigma}\|T\|_F+\frac{1}{\sigma}\sqrt{2\sigma\lambda_*+61 \|T\|^2_F}$. In particular, the latter is achieved if $\|s^*\|\geq \max\left\{\frac{32}{\sigma}\|T\|_F, 6\left(\frac{\lambda_*}{\sigma}\right)^{1/2}\right\}$. We now recall the lower bound in
 \eqref{bound for s thm}, namely,  $ \|s^*\|\ge \big(\frac{\|g\|}{3 \sigma} \big)^{1/3}$, which we easily check that it satisfies $\big(\frac{\|g\|}{3 \sigma} \big)^{1/3}\geq \max\left\{\frac{32}{\sigma}\|T\|_F, 6\left(\frac{\lambda_*}{\sigma}\right)^{1/2}\right\}$ when $\sigma$ satisfies \eqref{bound for sigma B psd_new}.  }}

\end{proof}

The bound in \eqref{bound for sigma B psd_new} contains \(\|g\|\) in the denominator. 
When the quartically regularized polynomial is used as a subproblem within a higher-order method, 
the stopping condition is typically chosen as \(\|g\| \ge \epsilon_g\), meaning that the algorithm 
terminates once an approximate critical point of the subproblem has been identified.  
As shown in \cite{cartis2020concise}, even an approximate critical point is sufficient to guarantee 
convergence and complexity of the outer algorithm.  
Therefore, the bound in \eqref{bound for sigma B psd_new} is well defined in the context of high-order optimization methods.  
For completeness, we also discuss the behaviour in the degenerate case \(\|g\| = 0\) 
in \Cref{remark what if g=0}.

\begin{remark} We give a discussion on what would happen if \(\|g\| = 0\).
\begin{itemize} \setlength{\itemindent}{0pt}
    \item If \(H  \succeq \delta I_n \succ 0\), then \(m_3(s) - m_3(s^*)\) is both SoS and SoS-convex for any 
    \(\sigma > \tfrac{8\|T\|^2}{\delta}\); see \Cref{sec: Locally convex m3}.
    
    \item If \(T = 0\), or if \(T\) has special structure
     \(T[s]^3 = t^Ts\, \|s\|^2\),  
    then \(m_3(s) - m_3(s^*)\) is SoS for every \(\sigma > 0\); 
    see \Cref{sec special T}.
    
    \item If \(H\) is not PSD and \(T\) is general, 
    determining whether \(m_3(s) - m_3(s^*)\) becomes SoS for sufficiently 
    large \(\sigma > 0\) remains an open question. 
    We leave this analysis to future work.
\end{itemize}
\label{remark what if g=0}
\end{remark}



Beyond the behavior of \(B(s^*)\), 
we also remark that if \(s^*\) is any second-order local minimizer of the model \(m_3(s)\), 
then \(\|s^*(\sigma)\|\) is a non-increasing function of \(\sigma\). 

\subsubsection{The Locally Convex Regime: Global Convexification and SoS Convex}
\label{sec: Locally convex m3}


If \(m_3\) is locally strongly  convex at \(s = 0\), i.e., 
\( H=\nabla^2 m_3(0) \succeq \delta I_n \succ 0\), 
then for sufficiently large \(\sigma\), the model \(m_3\) remains strongly convex and is, in fact, SoS convex for all \(s \in \mathbb{R}^n\) \cite{ahmadi2023higher}. 
A detailed discussion on how convexification relates to SoS and SoS-convexity 
can be found in~\cite[Sec.~2.3.2]{zhu2025global}. 
Here we highlight the key arguments.

If \(m_3\) is strongly convex at \(s = 0\), then ~\cite[Lemma~2.3]{zhu2025global} shows that choosing \(\sigma\) such that
\begin{eqnarray}
    \sigma > \frac{\|T\|^2}{4\delta}
    \label{sigma bound local convex}
\end{eqnarray}
guarantees \(m_3\) is strongly convex for all \(s\in\mathbb{R}^n\). 
A sufficient condition for a unique global minimizer is the global strong convexity of $m_3$, i.e., $\nabla^2 m_3(s)\succeq \delta I_n$ for all $s$ and some $\delta >0$. Then, \eqref{sigma bound local convex} ensures $m_3$ has a unique global minimizer. 
Moreover, if we increase \(\sigma\) further so that
\begin{eqnarray}
\sigma > \frac{8\|T\|^2}{\delta},
    \label{sigma bound SoS convex}
\end{eqnarray}
then by~\cite{ahmadi2023higher}, \(m_3\) is SoS-convex. Then, it follows from \Cref{lemma SoS convex} 
that \(q(s) = m_3(s) - m_3(s^*)\) is SoS.

\subsubsection{The Locally Nonconvex Regime: Global Optimality without Convexity}
\label{sec: Locally nonconvex m3 necc suff}

The size of $\sigma$ is closely related to the idea of convexification of $m_3$ with large $\sigma$. 
If \(m_3\) is locally nonconvex at \(s = 0\), then no matter how large \(\sigma > 0\) is chosen, it is not possible to make \(m_3\) globally convex for all \(s \in \mathbb{R}^n\). 
However, by selecting \(\sigma\) sufficiently large, \(m_3\) becomes convex everywhere except within a small neighborhood of the origin. 
In this section, we show that for sufficiently large $\sigma$, {the stationary points (may be non-unique)} lie outside the small locally nonconvex region and are global minimizers of the nonconvex model $m_3$. Moreover, in this regime, the global necessary and sufficient optimality 
conditions established in \cite{zhu2025global} coincide (more details are discussed~\cite[Sec 2.3.1]{zhu2025global} and Appendix \ref{appendix global opt}). \(m_3(s)\) is SoS exact with zero levels of relaxation. If the global minimizer is unique, the minimum can then be found in polynomial time.


\begin{lemma} (Lemma 2.4 from \cite{zhu2025global})
If $\lambda_{\min}[H] \leq 0$, let $s_0 > 0$ be any scalar, and suppose that $\sigma$ satisfies
\begin{eqnarray}
\sigma > 2 \max \bigg\{\lambda_* s_0^{-2}, \quad  \|T\|  s_0^{-1}\bigg\},
\label{convex sigma bound DTM}
\end{eqnarray}
then $m_3(s)$ is convex for all $\|s\| \geq s_0$.  
\label{lemma Convexification nonconvex}
\end{lemma}

\begin{corollary}
Assume that \(m_3\) is locally nonconvex at \(s = 0\), 
that is, \(H = \nabla^2 m_3(0)\) is not positive definite. 
Let \(s^* \in \mathbb{R}^n\) satisfies
$\nabla m_3(s^*)=0$. 
If $\sigma$ large enough such that~\eqref{sigma bound for necc suff} holds, 
then \(m_3(s)\) is convex for all \( \|s\| \geq \sqrt{\frac{2}{3}}\,\|s^*\| \).
\end{corollary}

\begin{proof}
Suppose that \eqref{sigma bound for necc suff} holds. 
Then, by Result~1 in \Cref{thm nece = suff}, we have
\[
\sigma \ge 2 \max \bigg\{
\lambda_* \Big\| \sqrt{\frac{2}{3}}\, s^* \Big\|^{-2}, 
\|T\| \, \Big\| \frac{2}{3} s^* \Big\|^{-1}
\bigg\}
>
2 \max \bigg\{
\lambda_* \|s_0\|^{-2}, 
\|T\| \, \|s_0\|^{-1}
\bigg\},
\]
where \(s_0 = \sqrt{\frac{2}{3}}\, s^*\). 
Therefore, by \Cref{lemma Convexification nonconvex}, 
\(m_3(s)\) is convex for all 
\(\|s\| \ge \|s_0\| 
= \sqrt{\frac{2}{3}}\, \|s^*\|\).
\end{proof}

\begin{remark}[Region of convexification]
If $\sigma$ is sufficiently large so that~\eqref{bound for sigma B psd_new} holds, then any stationary point $s_*$ of $m_3$  {{ with $m_3(s_*)\leq m_3(0)$}} is a global minimizer (\Cref{thm B psd} and \Cref{thm: certify SoS Quartically Regularized Polynomial}). In this regime, $m_3$ is nonconvex only in a small neighborhood of the origin and is convex outside the region
$
\{\, s \in \mathbb{R}^n : \|s\| \le \|s_0\| = \sqrt{\tfrac{2}{3}}\,\|s^*\| \,\}.
$
{Moreover, any global minimizer lies outside this nonconvex region, and no stationary point {{ with $m_3(s_*)\leq m_3(0)$}} exists within it. This follows because, when \eqref{sigma bound for necc suff} holds, every stationary point of $m_3$ {{ with $m_3(s_*)\leq m_3(0)$}} is a global minimizer, and the global minimizer necessarily lies in the convex region, implying that no stationary point {{ with $m_3(s_*)\leq m_3(0)$}} can exist within it.}
\end{remark}

\begin{remark}(Condition for uniqueness of the global minimizer)
As shown in this section, there are no stationary points {{ with $m_3(s_*)\leq m_3(0)$}} inside the inner nonconvex region.
To establish the uniqueness of the global minimizer, one must additionally guarantee
that the stationary point outside the nonconvex region is unique.
In that case, this unique stationary point is also the unique global minimizer.
Without such additional conditions, we can ensure existence of a global minimizer,
but not necessarily uniqueness.
\end{remark}

\begin{remark} (Comparing different $\sigma$ bound)
{\cite[Sec.~2.3.1]{zhu2025global} (also summarized in this paper’s notation in \Cref{thm nece = suff}) provides a bound under which the necessary and sufficient conditions coincide.}
The bounds in \Cref{thm nece = suff} and \Cref{thm B psd} are comparable in magnitude; 
both depend only on the coefficients of \(m_3\), 
and both are of order\footnote{%
In iterative algorithms for minimizing \(f\), 
We solve a sequence of cubic subproblems \(m_3\), 
with a stopping criterion \(\|g\| \ge \epsilon_g\). 
Therefore, in both bounds we have \(\sigma = O(\epsilon_g^{-2})\).} \(\mathcal{O}(\|g\|^{-2})\).  
In the locally nonconvex regime, a stationary point that achieves a decrease in the model value 
is the global minimizer, and the necessary and sufficient conditions coincide.
Since \(m_3(s) - m_3(s^*)\) is a SoS polynomial, 
if the minimizer is unique, the global minimization of \(m_3\) 
can be reformulated as a semidefinite program (SDP), 
from which the global minimum can be recovered in polynomial time.
\end{remark}

\section{Structured Classes of Polynomials with Exact SoS Certificates}
\label{sec: Special Classes of SoS Quartically Regularized Polynomial}

{

The theory developed in Section~\ref{sec: main theory} shows that
the shifted quartically regularized cubic model becomes SoS once the regularization
parameter $\sigma$ is sufficiently large. A natural next question is whether this
large-regularization requirement  can be removed in some cases, for some choices of $g$, $H$ and $T$ in \eqref{m3}.

This section answers this question in several directions. We identify special
classes of quartically regularized polynomials for which the shifted model
\[
q(v)=m_3(s^*+v)-m_3(s^*) \text{ is SoS for \emph{any} $\sigma>0$,}
\]
where $s^*$ denotes a global minimizer. 
These examples show that the phenomenon established in the previous section is not just asymptotic: in certain structured subclasses, SoS exactness is present for all regularization levels.

At the same time, we also show that this behavior is not universal for all regularization norm. In particular, for the classic Schnabel tensor model in 1991~\cite{schnabel1991tensor}, whose regularization is based on a separable quartic norm, the analogue of our result fails in general, even as $\sigma\to\infty$. 
This contrast helps explain why Euclidean quartic regularization plays a special
role in the SoS structure of quartically regularized models.

The section is organized around three positive examples and one negative one:
the univariate case, the quadratic-quartic model ($T=0$), the special tensor model
$T[s]^3=(t^Ts)\|s\|^2$, and finally the Schnabel model. A summary is given in
Table~\ref{table:SoS_conditions_summary}.
}


\begin{table}[h]
\centering
\caption{\small Summary of SoS conditions for shifted quartically regularized polynomials. The table below should be read as a map of the section: it highlights which structural assumptions remove the need for large regularization, and which do not.}
\label{table:SoS_conditions_summary}
\renewcommand{\arraystretch}{1.25}
\setlength{\tabcolsep}{6pt}

\small
\begin{tabular}{p{3.8cm} p{7cm} p{4.5cm}}
\hline
\textbf{Polynomial Type} 
&
\textbf{Condition for $q(v)$ to be SoS} 
&
\textbf{SoS exactness conclusion}
\\ \hline

General $m_3$ 

(\Cref{sec geneeral B psd})
&
$\nabla m_3(s^*) = 0$ and $B(s^*) \succeq 0$ \eqref{eq:B-sstar}
&
Yes, if $\sigma$ is large enough. 
\\ \hline

Univariate $m_3$

(\Cref{sec n =1})
&
$\nabla m_3(s^*) = 0$ and 
$H + \frac{1}{3}T[s^*] + \sigma\|s^*\|^2 
\ge \frac{T^2}{18\sigma}$
&
Yes, for all  $\sigma>0$.
\\ \hline

QQR model ($T=0$)

\eqref{QQR Model}
&
$\nabla m_3(s^*) = 0$ and 
$H + \sigma\|s^*\|^2 I_n \succeq 0$
&
Yes, for all  $\sigma>0$.
\\ \hline

Special $T$ model

$T[s]^3=t^T s\|s\|^2$

(\Cref{sec special T})
&
$\nabla m_3(s^*) = 0$ and 
$B_T(s^*) \succeq 0$ 
\eqref{opt cond n special T}
&
Yes for all  $\sigma>0$.
\\ \hline

Schnabel Model 

(\Cref{sec Schnabel Model})
&
$\nabla m_\mathcal{S}(s^*) = 0$ and 
$B_\mathcal{S}(s^*) \succeq 0$
\eqref{optimality condition sch}
&
Sometimes no, even if $\sigma$ is large. 
\\ \hline
\end{tabular}
\end{table}

\normalsize

\subsection{The Univariate Case: Exactness without Regularization}
\label{sec n =1}

We begin with the one-dimensional case as a useful benchmark. 
Here the large-regularization threshold from the general theory is no longer needed:
once $s^*$ is a global minimizer, the shifted polynomial is automatically SoS for every
$\sigma>0$. This is consistent with Hilbert's classical result that every nonnegative
univariate polynomial is a sum of squares, but the derivation below is tailored to
the quartically regularized model and reveals the relationship with global optimality condition.

\begin{lemma}
   For $n=1$, $s^*$ satisfies
  \begin{eqnarray}
      \nabla m_3(s^*) = 0 \qquad \text{and} \qquad
    H + \frac{1}{3}T[s^*] + \sigma \|s^*\|^2 \ge  \frac{T^2}{18  \sigma},
    \label{opt cond n=1}
  \end{eqnarray}
if and only if $s^*$ is a global minimizer. 
\label{lemma opt cond n=1}
\end{lemma}

\begin{proof}
The proof is provided in Appendix~\ref{appendix proof for  univariate case}. 
\end{proof}

 Hilbert \cite{hilbert1888darstellung} proved that every nonnegative univariate polynomial is SoS. 
Here we provide an alternative derivation of this result using global optimality conditions.

\begin{theorem}
    For $n=1$, if $s^*$ is a global minimizer, then $q(v) =  m_3(s^*+v) - m_3(s^*)$ is SoS. 
    \label{thm: SoS in n=1}
\end{theorem}

\begin{proof}
The proof is provided in Appendix~\ref{appendix proof for  univariate case}. 
\end{proof}

\begin{remark}
Consider global optimality condition in \cite[Thm 2.1 and 2.2]{zhu2025global}, for $n=1$, \eqref{necessary tight} gives to  
$   
H + \frac{1}{3} T [s^*] +\sigma  \|s^*\|^2    \ge -\frac{1}{3} T [s^*]   + \frac{|T|}{3} \|s^*\| \ge 0
$
and \eqref{sufficient tight}  gives to    
$  
H + \frac{1}{3} T [s^*] +\sigma  \|s^*\|^2    \ge \frac{{\hat{T}}^2}{18\sigma}  -\frac{1}{3} T [s^*]   + \frac{|T|}{3} \ge \frac{{\hat{T}}^2}{18\sigma}.
$
note that \eqref{opt cond n=1} lies between \eqref{necessary tight} and \eqref{sufficient tight}, such that
$$
\eqref{sufficient tight} \Rightarrow \eqref{opt cond n=1} \Rightarrow\eqref{necessary tight}.
$$
\Cref{lemma opt cond n=1} provides a tighter necessary and sufficient condition for global optimality in the $n=1$ case. 
\end{remark}

{Thus, in the univariate setting, the quartically regularized model exhibits SoS exactness
for all $\sigma>0$. The next examples show that this phenomenon persists in certain
higher-dimensional structured cases, but not for arbitrary tensors.}

\subsection{A Structured Tensor Model with SoS Exactness}
\label{sec special T}

{

We next turn to a multivariate setting in which SoS exactness again holds
for every $\sigma>0$. The key structure is that the cubic tensor term has the special form
\[
T[s]^3=(t^Ts)\|s\|^2,
\]
so that the quartically regularized model becomes
\begin{equation}
\label{m3 special T}
m_{ST}(s)
= f_0 + g^T s + \frac{1}{2} H[s]^2
+ \frac{1}{6}\, (t^Ts)\|s\|^2
+ \frac{\sigma}{4}\, \|s\|^4.
\tag{Special T Model}
\end{equation}
This model is of particular interest because it represents a highly structured
low-complexity family of cubic tensors, parameterized by a single vector $t$,
rather than a generic third-order tensor. In this sense, it may be viewed as a
structured low-rank-type tensor model. Such tensor forms also arise naturally in
higher-order optimization and tensor methods. This makes the model relevant not only as a theoretical example, but also as a
tractable surrogate in tensor-based regularization methods and structured third-order models. 
The structural alignment between the cubic term $(t^Ts)\|s\|^2$ and the Euclidean
quartic regularization $\|s\|^4$ is what makes this case special: it allows the
mixed terms to be completed into a single square, and hence removes the need for
a large-regularization threshold.
}

{The next lemma shows that in this structured tensor class, the sufficient and
necessary global optimality conditions coincide. The subsequent theorem then shows
that the shifted polynomial is SoS whenever $s^*$ is a global minimizer.}

\begin{lemma} If for any $s \in \R^n$, $T$ satisfies $T[s]^3 = t^Ts \|s\|^2$ for some $t \in \R^n$ then $s^*$ satisfies
\begin{eqnarray}
  \nabla m_{ST}(s^*) = 0 \qquad \text{and} \qquad
B_T(s^*):= H +  T[s^*] + \sigma \|s^*\|^2 I_n -\frac{1}{3} ({s^*} t^T+t {s^*}^T) - \frac{t t^T}{18  \sigma}   \succeq 0,
\label{opt cond n special T}  
\end{eqnarray}
if and only if $s^*$ is a global minimizer of $m_{ST}$. 
\label{lemma opt cond special T}
\end{lemma}

\begin{proof} 
Let $q(v):=m_{ST}(s^*+v) - m_{ST}(s^*)$. Using the expression of $B_T(s^*)$, we have
$$
\nabla^2 m_{ST}(s^*) = B_T(s^*) +2\sigma s^*  {s^*}^T  + \frac{1}{3}  ({s^*} t^T+t {s^*}^T) + \frac{t t^T}{18  \sigma} I_n.
$$
Substituting into  \eqref{expression m3 using m_*}, we have
\begin{eqnarray*}
2 \big[q(v)- \nabla m_{ST}(s^*)^T  v \big] &=&
 v^T \bigg[ \underbrace{ B_T(s^*) +2\sigma s^*  {s^*}^T  + \frac{1}{3}  ({s^*} t^T+t {s^*}^T) + \frac{t t^T}{18  \sigma} I_n }_{=\nabla^2 m_{ST}(s^*)} +  \frac{1}{3} T[v] + 2 \sigma {s^*}^T  v  + \frac{\sigma}{2} \|v\|^2 I_n\bigg]v 
\\&=&  B_T(s^*)[v]^2 +2\sigma (v^Ts^*)^2 + \frac{2}{3}(v^T {s^*}) (t^T v)  + \frac{(v^T t)^2}{18  \sigma}+  \frac{1}{3} (t^T v)\|v\|^2 + 2 \sigma {s^*}^T  v  \|v\|^2  + \frac{\sigma}{2} \|v\|^4 
\\&=&  B_T(s^*)[v]^2 + \frac{\sigma}{2} \bigg(\|v\|^2+ 2v^Ts^* + \frac{v^Tt}{3\sigma}  \bigg)^2
\\&=&  B_T(s^*)[v]^2 + \frac{\sigma}{2} \left[ \bigg\|v + \big(s^* + \frac{t}{6\sigma}\big)\bigg\|^2- \bigg\|s^* + \frac{t}{6\sigma}\bigg\|^2\right]^2.
\end{eqnarray*}

`$\Rightarrow$'  
Clearly, if \eqref{opt cond n special T} is satisfied, $m_{ST}(s^*+v) - m_{ST}(s^*) \ge 0$ for all $v$. Therefore, $s^*$ is the global minimum.

`$\Leftarrow$' If $s^*$ is a global minimum, then it is a stationary point, therefore $      \nabla m_{ST}(s^*) =0$. 
Let $w := v+ 2s^* + \frac{t}{3\sigma}$. Assume $w^T v=0$, since $s^*$ is the global minimum, we have $m_{ST}(s^*+v) - m_{ST}(s^*)  \ge 0$ for all $v \in \R^n$, thus $ B_T(s^*) [v]^2 \ge 0$ for $v$ orthogonal to $w$. 

Otherwise, assume  $w := v+ 2s^* + \frac{t}{3\sigma}$ and $s_c:=s^* + \frac{t}{6\sigma}$ are not orthogonal, the line $s_c + \tilde{k} w$ intersects the ball centred at the origin of radius $\|s_c\|$  at two points, $s_c$ and $u_*$, with
$ \|u_*\| = \|s_c\|.$ We set $v = - s_c+  u_* = - \big(s^* + \frac{t}{6\sigma}\big) +  u_*$.  Since $s^*$ is the global minimum, we have $m_{ST}(s^*+v) - m_{ST}(s^*)  \ge 0$ for all $v \in \R^n$, we arrive at $ B_T(s^*) [v]^2 \ge 0$ for $v$ not orthogonal to $w$. 
\end{proof}

\begin{theorem}
    For $n \ge 2$, assume for any $s \in \R^n$, $T$ satisfies $T[s]^3 = t^Ts \|s\|^2$ for some $t \in \R^n$. If $s^*$ is the global minimizer, 
then $q(v) =  m_{ST}(s^*+v) - m_{ST}(s^*)$ is SoS. 
\label{thm: SoS in special T}
\end{theorem}
\begin{proof}
 Since $s^*$ is the global minimum, we have $\nabla m_{ST}(s^*) =0$, according to \eqref{expression m3 using m_*}, we can write
\begin{eqnarray}
 m_{ST}(s^*+v)-m_{ST}(s^*)  = \frac{1}{2}
\omega
 \underbrace{\begin{bmatrix} 
 0 &  0  &  0 & 0 &
\\0 &  \nabla^2  m_{ST}(s^*) &  (\sigma s^* + \frac{t}{6}) \mathbf{1}_n^T   & 0
\\0 &  \mathbf{1}_n(\sigma s^* + \frac{t}{6} )^T & \frac{\sigma}{2} \mathbf{1}_n\mathbf{1}_n^T & 0
\\0 &  0 & 0 & 0
 \end{bmatrix}}_{Q_c}
\omega
\label{matrix form Q}
\end{eqnarray}
where $\phi_2$ given in \Cref{phi2}, 
$\mathbf{1} = [1, \dotsc, 1]^T \in \R^n$ and $\mathbf{1}\mathbf{1}^T \in \R^{n \times n}$ is a matrix with all entries equal to 1s. 
Let $\omega \in \R^{\frac{1}{2}(n+2)(n+1)}$ be any vector, we write $\omega  = [1,v, u, z]^T$ as given in \Cref{phi2}. 
For any $\omega$, or equivalently for any $(u, v, z)$, we have
\begin{eqnarray}
2 \omega^T  Q_c(v)\omega   
&=& \nabla^2  m_{ST}(s^*) [v]^2 + 2v^T(\sigma s^* + \frac{t}{6} ) \sum_{i=1}^n u_i  + \frac{\sigma}{2}  \sum_{i=1}^nu_i^2 +  \sigma \sum_{1 \le i < j \le n} u_iu_j 
\notag 
\\  &=& \nabla^2  m_{ST}(s^*) [v]^2 + 2v^T(\sigma s^* + \frac{t}{6} ) e^T u + \frac{\sigma}{2} (e^Tu)^2 
\notag
\\ &=& \underbrace{\nabla^2  m_{ST}(s^*) [v]^2 - 2 \sigma \bigg[ v^T(s^* + \frac{t}{6\sigma} ) \bigg]^2 }_{: \mathcal{J}_1}+ \frac{\sigma}{2} \bigg[\sum_{i=1}^n u_i + 2v^T(s^* + \frac{t}{6\sigma} )  \bigg]^2 . \label{SoS intermid special t}
\end{eqnarray}
Since $s^*$ is the global minimum, \Cref{lemma opt cond special T} proves that \eqref{opt cond n special T} holds. Rearranging \eqref{opt cond n special T} gives $\nabla^2  m_{ST}(s^*) [v]^2  \ge \frac{1}{3} ({s^*} t^T+t {s^*}^T) + \frac{t t^T}{18  \sigma} + 2\sigma s^*  {s^*}^T $ for all $v \in \R^n$. Consequently, 
\begin{eqnarray*}
2 \omega^T  Q_c(v)\omega   \ge \mathcal{J}_1 \ge  v^T \bigg[\frac{1}{3} ({s^*} t^T+t {s^*}^T) + \frac{t t^T}{18  \sigma}  I_n + 2\sigma s^*  {s^*}^T  \bigg] v - 2 \sigma v^T \bigg[(s^* + \frac{t}{6\sigma })  ( s^* + \frac{t}{6\sigma })^T   \bigg] v \ge 0
\end{eqnarray*}
for all $v \in \R^n$. Thus, we obtained that $Q_c \succeq 0$ and therefore $ m_{ST}(s^*+v)- m_{ST}(s^*)$ is SoS. 
\end{proof}

\begin{remark} (Extending to $W$-norm) By changing $I_n$ to $W$-norm with $W \succeq 0$, if the model is bounded below or $W \succ 0$, then \Cref{thm: SoS in special T} and \Cref{lemma opt cond special T} can be extended to allow $T[s]^3 = (t^Ts) (s^TWs)$ by using the change of $\tilde{s} = \hat{D}^{1/2}\hat{L}s$ where $\hat{L}^T\hat{D}\hat{L}=W$.  
\label{extending to W norm}
\end{remark}

{

\paragraph{The quadratic-quartic model as a limiting case.}
The special tensor model reduces to the quadratic-quartic regularization (QQR) model given in \cite{cartis2023second}
in the limiting case $t=0$, i.e., $T=0$,} 
\begin{equation}
\label{QQR Model}
m_{QQR}(s)
= f_0 + g^T s + \frac{1}{2} H[s]^2 
+ \frac{\sigma}{4}\, \|s\|^4.
\tag{QQR Model}
\end{equation}
In this case, the global optimality condition
reduces to
\begin{eqnarray}
\label{qqr opt}
\nabla  m_{QQR}(s^*)=0,
\qquad
H+\sigma\|s^*\|^2I_n \succeq 0,
\end{eqnarray}
and the shifted polynomial is SoS for every $\sigma>0$.  \Cref{lemma opt cond special T} can be viewed as a generalization of the global optimality result in \cite[Theorem 8.2.8]{cartis2022evaluation}, which is recovered by setting $t=0$. 

\begin{remark}
\textbf{A discussion on the bivariate case.}
{When $n=2$, the shifted polynomial $q(v)$ is a nonnegative bivariate quartic,
and is therefore automatically SoS by Hilbert's classical theorem (i.e., every nonnegative bivariate quartic polynomial is SoS \cite{hilbert1888darstellung}.)  
This suggests that the special tensor model above may be viewed as one instance
of a broader structural phenomenon: certain low-complexity cubic terms may admit
exact SoS certificates even without large regularization.}
\end{remark}

\subsection{Beyond Euclidean Regularization: When SoS Exactness Breaks Down}
\label{sec Schnabel Model}

{
{
The previous examples might suggest that nonnegative quartic terms are key to ensuring SoS exactness for $m_3-m_*$. 
However, not all regularization norms are suitable. In particular, the classical Schnabel model (1991) has a  regularization term that is based on a separable quartic norm rather than the Euclidean
norm, and this seemingly minor change fundamentally alters the algebraic structure of the ensuing polynomial.
We will show in this section that one can still derive sufficient SoS-based optimality conditions for the Schnabel model. However, unlike the Euclidean case, these conditions do not become automatically satisfied for large regularization in general. The Schnabel model provides a useful counterexample: it illustrates that the form of the regularization, not merely its degree, is crucial.}}

Schnabel et al. in 1991 \cite{schnabel1991tensor} proposed a formulation for unconstrained optimization based on a fourth-order model of the objective function $f$. The model is given by
\begin{equation}
\tag{Schnabel Model}
m_{\mathcal{S}}(s):= f_0 + g^T s +\frac{1}{2}H[s]^2 + \frac{1}{6} \sum_{j=1}^k (a_j^Ts)^2(b_j^Ts) + \frac{1}{4}\sum_{j=1}^k {\sigma_j}(a_j^Ts)^4 
\label{sch}
\end{equation}
where $\sigma_j >0$, $a_j, b_j \in \R^n$ are specific vectors which correspond to the low–rank approximations of the third and fourth-order derivatives and are built by interpolating function and gradient information from past iterates.  
Schnabel et al. showed that although \eqref{sch} is a fourth-order polynomial in $n$ variables, it can be reduced to the minimization of a fourth-order polynomial in $k$ variables plus a quadratic in $n-k$ variables. Note that the quartic term in \(m_{\mathcal{S}}\) does not necessarily ensure that 
\(m_{\mathcal{S}}(s) \to +\infty\) as \(\|s\|\to\infty\) in all directions. 
Therefore, the Schnabel model may not be bounded below 
and cannot guarantee the existence of a global minimizer of the model function. 
In \cite{schnabel1991tensor, schnabel1984tensor}, Schnabel et al. used this model within a trust-region framework to obtain global convergence of the overall algorithm.

\paragraph{Separable Regularization Norm} 
To ensure that \eqref{sch} is bounded below. Set \(k = n\) and, without loss of generality, assume that \(\sigma_j = \sigma\) and \(a_j = e_j\), where \(e_j\) denotes the unit vector with 1 in the \(j\)-th entry and 0 elsewhere. Note that any other choice of \(\sigma_j, a_j\) can be reduced to this form through an appropriate change of variables.  Let $T = \sum_{j=1}^n e_j \circ e_j \circ b_j$, \eqref{sch} simplifies to
\begin{equation*}
\begin{aligned}
m_{\mathcal{S}}(s) &:= f_0 + g^T s + \frac{1}{2} H[s]^2 + \frac{1}{6} T[s]^3 + \frac{\sigma}{4} \|s\|_4^4, \qquad\text{where}\qquad \|s\|_4 = \big(\sum_{j=1}^n s_j^4 \big)^{1/4},  
\end{aligned}
\end{equation*}
The key difference between this model and the quartically regularized polynomial model \(m_3\) lies in the regularization term: here we employ a \emph{quartically separable norm} \(\|s\|_4^4 = \sum_{j=1}^n s_j^4\), whereas \(m_3\) uses a \emph{quartically regularized Euclidean norm} $\|s\|_2^4 = (s_1^2 + \dotsc s_n^2)^2$. 

{The next result shows that an SoS certificate can still be constructed for the Schnabel model,
but only under an explicit matrix condition. In contrast to the Euclidean case, this condition
does not become automatic for large regularization in general.} \Cref{thm: SoS sch} proves the sufficient conditions for global optimality and its relationship to SoS characterization. \Cref{opt for sch} proves the necessary conditions for global optimality. 
Previously, in \Cref{lemma opt cond special T}, we established tight necessary and sufficient conditions that fully characterize the global minimizer for $m_3$.  
For \eqref{sch}, however, there is an inherent gap between necessary and sufficient 
conditions.



\begin{theorem} (SoS and Global Sufficient Optimality)
    For $n \ge 2$, assume $s^*$ satisfies 
    \begin{eqnarray}
      \nabla  m_{\mathcal{S}}(s^*) = 0 \qquad \text{and} \qquad
 B_{\mathcal{S}}(s^*):=  \nabla^2  m_{\mathcal{S}}(s^*)  - 2 \sum_{i=1}^k \sigma_j \bigg[ (a_ja_j^T) s^*  + \frac{b_j}{6\sigma_j} \bigg]  \bigg[ (a_ja_j^T) s^*  + \frac{b_j}{6\sigma_j} \bigg]^T \succeq 0,
    \label{optimality condition sch}
    \end{eqnarray}
    then $q_{\mathcal{S}}(v) =   m_{\mathcal{S}}(s^*+v) -  m_{\mathcal{S}}(s^*)$ is SoS. Also, if $s^*$ satisfies
\eqref{optimality condition sch}, 
then, $s^*$ is the global minimum. 
\label{thm: SoS sch}
\end{theorem}

\begin{proof}
We provide the proof in Appendix~\ref{appendix Proof for Schnabel}. 
\end{proof}

\begin{remark}

 Using the expression of $\nabla^2 m_{\mathcal{S}}$, the alternative formulation of $ B_{\mathcal{S}}(s^*)$ is
\begin{eqnarray}
 B_{\mathcal{S}}(s):=   H + T[s] + \sum_{j=1}^k \sigma_j  (a_j a_j^T) (a_j^Ts)^2  - \sum_{j=1}^k \frac{a_ja_j^T}{3}\bigg[ s b_j^T+b_j {s}^{T} \bigg] - \sum_{j=1}^k \frac{b_jb_j^T}{18 \sigma_j}
  \label{B_s star}
\end{eqnarray}
where  $T = \sum_{j=1}^k a_j \circ a_j  \circ  b_j$. This condition can be viewed as a generalization of~\eqref{opt cond n special T}. From \Cref{thm: SoS in n=1}--\Cref{thm: SoS sch}, we observe that the SoS conditions are derived in more general forms. For \(k = n\) and \(a_j = e_j\), where \(e_j\),      $B_{\mathcal{S}}(s) $ has the expression, 
\begin{eqnarray}
B_{\mathcal{S}}(s) := H + T[s] + \sigma\,\mathrm{diag}\{s_j^2\}_{1 \le j \le n}
- \sum_{j=1}^n \frac{e_j e_j^T}{3} \big( s b_j^T + b_j s^T \big)  - \sum_{j=1}^n \frac{b_j b_j^T}{18\sigma_j}.
\label{B_s star2}
\end{eqnarray}
\end{remark}

\begin{theorem}
(Global Necessary Optimality) Let $\text{Im}(A) := \operatorname{span} \{a_1,\dotsc, a_k\}$ and its orthogonal complement space $\text{Im}(A)^\perp := \operatorname{span} \{a_{k+1},\dotsc, a_n\}$.  {If a global minimizer $s^*$ exists and is attained}, then $\nabla  m_{\mathcal{S}}(s^*) = 0$ and
\begin{eqnarray*}
B_{\mathcal{S}}(s^*)[v]^2 \ge 0  \qquad \text{for all} \qquad v \in \text{Im}(A)^\perp.
\end{eqnarray*}
If $ \{a_1,\dotsc, a_k\}$ are orthogonal vectors, then we also have 
\begin{eqnarray*}
B_{\mathcal{S}}(s^*)[a_j]^2 \ge 0  \qquad \text{for all} \qquad a_j \in \text{Im}(A).
\end{eqnarray*}
\label{opt for sch}
\end{theorem} 
\begin{proof}
We provide the proof in Appendix~\ref{appendix Proof for Schnabel}. 
\end{proof}


A natural question arises: is \(B_{\mathcal{S}}(s^*)\) positive semidefinite when \(\sigma\) is sufficiently large?  
Equivalently, if \(\sigma\) is large enough, can \(m_{\mathcal{S}}(s) - m_{\mathcal{S}}(s^*)\) be represented as a SoS?  
Unfortunately, unlike the case of \(m_3\), the answer is \emph{no}. An illustrative counterexample is provided in \Cref{sec: Ahmadi Example}.  

The key geometric distinction is that Euclidean regularization contributes a uniform 
positive shift $\sigma\|s^*\|^2 I_n$ to $B(s^*)$, whereas separable regularization 
contributes only $\sigma \mathrm{diag}\{s_j^2\}$, whose smallest eigenvalue is controlled 
by $\min_j |s_j|$ rather than $\|s^*\|$---a quantity that cannot be bounded by the 
step-size estimates.

The contrast with the Euclidean model is now clear. In the quartically regularized
cubic model $m_3$, the regularization contributes the uniform positive term
$\sigma\|s^*\|^2 I_n$, which can dominate the indefinite contributions in $B(s^*)$.
In the Schnabel model, by contrast, the regularization contributes only the diagonal term
\[
\sigma\,\mathrm{diag}\{s_j^2\}_{1\le j\le n},
\]
whose smallest eigenvalue depends on the smallest component of $s^*$ rather than on
its full norm. Specifically, in \(m_3\), the regularization term \(\sigma \|s^*\|^2 I_n\) has minimum eigenvalue equals \(\sigma \|s^*\|^2\).  
In contrast, the regularization term \(\sigma\,\mathrm{diag}\{s_j^2\}_{1 \le j \le n}\) has minimum eigenvalue is \(\sigma \min_{1 \le j \le n} \{ s_j^2 \}\).  
The bound $\|s^*\|=O(\|g\|^{1/3})$ therefore no longer suffices to force positive semidefiniteness. The step-size bound \(\|s^*\|\) does not directly control this minimum component \(\min_j |s_j|\), and hence does not guarantee that the regularization term in \(B_{\mathcal{S}}(s^*)\) dominates the remaining terms.  
Consequently, \(B_{\mathcal{S}}(s^*)\) may fail to be PSD.  This explains why SoS exactness can fail even for arbitrarily
large regularization. 
This demonstrates that even a subtle change in the form of the regularization norm can fundamentally alter the SoS properties of the model as \(\sigma \to \infty\).


\paragraph{Difference between the special $T$ model and the Schnabel model}
Under special cases, \eqref{sch} can be transformed into our special $T$ with a $W$ semi-norm ($W \succeq 0$). If $b_1 = b_2 = \dotsc = b_k$ and $\sigma_1 = \sigma_2 = \dotsc = \sigma_k=\frac{3}{2}\sigma$, we can set $ b_1:=\frac{t}{3}$ and  $ \sigma_1:=\sigma$, 
\begin{eqnarray}
\label{special sch case}
m_{\mathcal{S}}(s):= f_0 + g^T s +\frac{1}{2}H[s]^2 + \frac{1}{6} \bigg[s^T\underbrace{\big(\sum_{j=1}^k a_ja_j^T\big)}_{:=W}s\bigg] (t^Ts) + \frac{\sigma}{4} \|s\|_W^4.
\end{eqnarray}
Therefore, we have $T[s]^3 = (t^Ts)(s^TWs)$ and clearly $W \succeq 0$. Under this specific constructions, if the model is bounded below, then $q_{\mathcal{S}}(v)=m_{\mathcal{S}}(s^*+v) - m_{\mathcal{S}}(s^*)$ is SoS (see \Cref{lemma opt cond special T} and \Cref{{extending to W norm}}). 

\begin{remark}
\textit{(Difference between the regularization norm)}  
There is no matrix transform / linear change of basis that converts $ \sum_{j=1}^k s_j^4$ into the Euclidean norm $ \|s\|^4 = \big(\sum_{j=1}^k s_j^2 \big)^2$ for $j \ge 2$. 
More generally, the regularization term in \eqref{sch} takes the form
$
\sum_{j=1}^k (a_j^T s)^4 
= \sum_{j=1}^k (s^T a_j a_j^T s)^2
= \sum_{j=1}^k \|s\|_{W_j}^4, 
\qquad W_j = a_j a_j^T.
$
Here, $\|s\|_{W_j} := \sqrt{s^T W_j s}$ is the matrix-induced quadratic form.  
Since each $W_j$ is rank-one and positive semidefinite, $\|\cdot\|_{W_j}$ is not a norm on the whole space $\mathbb{R}^n$, but only a \emph{seminorm}: it vanishes for any $s \perp a_j$, even when $s \neq 0$.  
Only in the case $k=1$ does the regularization reduce to a single quartic seminorm, i.e.,
\[
(a_1^T s)^4 = \|s\|_{W_1}^4,
\]
which is in a form analogous to the quartic regularization norm in the special $T$ model (i.e., $\|s\|_W^4$ or $\|s\|^4$). The quartic semi-norm does not 
necessarily ensure that \(m_{\mathcal{S}}\) is bounded below. 
For instance, if \(s\in \operatorname{Im}(A)^\perp\) and 
\(H[s]^2 < 0\), then \(m_{\mathcal{S}}(s) \to -\infty\) as 
\(\|s\| \to \infty\). 
 \end{remark}

\begin{remark}
When $k=1$, 
$
m_{\mathcal{S}}(s)
:= 
f_0 + g^T s
+ \frac{1}{2} H [s]^2
+ \frac{1}{6} (a_1^T s)^2 (b_1^T s)
+ \frac{\sigma_1}{4} (a_1^T s)^4. 
$
In this case, the global necessary optimality condition in 
\Cref{opt for sch} becomes
\(\nabla m_{\mathcal{S}}(s^*) = 0\) and
\(B_{\mathcal{S}}(s^*) \succeq 0\).
Hence, if a global minimizer exists, the necessary and sufficient conditions coincide and \(m_{\mathcal{S}}(s) - m_{\mathcal{S}}(s^*)\) is SoS. This corresponds to the special case in 
\eqref{special sch case}, which can be transformed into the 
special \(T\)-model with a weighted quartic semi-norm. 
For $k>1$, however, the presence of multiple distinct seminorms $\|s\|_{W_j}$ prevents such a change of basis and reduction; the equivalence no longer holds.  
\label{remark diff between sch and special T}
\end{remark}

\subsubsection{Euclidean versus separable quartic regularization}
  \label{sec: Ahmadi Example}

{We now illustrate \Cref{opt for sch} and \Cref{thm: SoS sch} with an explicit example, showing that the sufficient 
SoS condition fails for separable quartic regularization even as $\sigma \to \infty$, 
while the Euclidean counterpart remains SoS-exact for all $\sigma > 0$. }
We illustrate using a modified example inspired by Ahmadi et al. in~\cite[Thm.~3.3]{ahmadi2023sums}. Consider a trivariate quartic polynomial,
\begin{equation}
\label{ex counterexample sch}
m_{\mathcal{A}}(s)
=
2\|s\|^2
+ 8(s_1s_2 + s_1s_3 + s_2s_3)
+ \frac{\sigma}{4} (s_1^4 + s_2^4 + s_3^4) = \hat{H}[s]^2 +  \frac{\sigma}{4} \sum_{i=1}^3 s_i^4,
\qquad \sigma > 0.
\end{equation}
where  $\hat{H} = [4,8,8;8,4,8;8,8,4] \in \R^{3 \times 3}$. $\nabla^2 m_{\mathcal{A}}(0) = \hat{H}$ which has eigenvalue $[-4, -4, 20]$. Clearly, $m_{\mathcal{A}}(s)$ is nonconvex for all $\sigma >0$. It is shown in~\cite[Thm.~3.3]{ahmadi2023sums} that, for $\sigma = 4$, 
the shifted polynomial is non-SoS. 
The following Lemma shows that this property persists for all $\sigma > 0$.

\begin{lemma}
For any $\sigma > 0$, let $m_{\mathcal{A}}(s)$ be defined as in~\eqref{ex counterexample sch}. Denote its global minimum value by $m_{\mathcal{A}}(s^*)$. Then, for every $\sigma > 0$, the polynomial \(m_{\mathcal{A}}(s) - m_{\mathcal{A}}(s^*)\) is nonnegative but not a SoS.
\label{lemma not SoS}
\end{lemma}

\begin{proof}
 For the specific case of $\sigma = 4$, we denote
$$
m_{\mathcal{A}, \sigma_0}(s) :=
2\|s\|^2
+ 8(s_1s_2 + s_1s_3 + s_2s_3) + (s_1^4 + s_2^4 + s_3^4)
= \hat{H}[s]^2 + \sum_{i=1}^3 s_i^4.
$$
Note that $m_{\mathcal{A}, \sigma_0}(s)$ is bounded below and possesses six global minimizers\footnote{Numerically, we obtain \(m^*_{(\sigma_0)} = -2.1443\), and $ s^*_{(\sigma_0)} = \pm(-1.1498,\, 0.6674,\, 0.6674)$ and their permutations.}. We denote the global minimizers and global minimum as $(s^*_{(\sigma_0)}, m^*_{(\sigma_0)})$.  
\cite[Thm 3.3]{ahmadi2023sums} proves  that, the shifted polynomial \(m_{\mathcal{A}, \sigma_0}(s) - m^*_{(\sigma_0)}\) is not SoS in \(s\). 
We prove that this result extends to all $\sigma > 0$.  
Take any fixed $\sigma > 0$, consider the change of variables $\tilde{s} = \frac{\sigma^{1/2}}{2} s$.  
Then, from~\eqref{ex counterexample sch}, we have
\begin{eqnarray*}
m_{\mathcal{A}}(s)
&=& \hat{H}[s]^2 + \frac{\sigma}{4}\sum_{i=1}^3 s_i^4 
=  4\sigma^{-1} 
\Bigg[\hat{H}\big[\frac{\sigma^{1/2}}{2}{s}\big]^2 + \sum_{i=1}^3 \big(\frac{\sigma^{1/2}}{2}{s}_i\big)^4 \Bigg]
= 4\sigma^{-1} m_{\mathcal{A}, \sigma_0}(\tilde{s}).
\end{eqnarray*}
By this scaling, \(m_{\mathcal{A}}(s)\) attains the global minimizers and global minimum at \((2\sigma^{-1/2}s^*_{(\sigma_0)}, 4\sigma^{-1}m^*_{(\sigma_0)})\).
As scaling preserve non-SoSness,  given \(m_{\mathcal{A}, \sigma_0}(\tilde{s}) - m^*_{(\sigma_0)}\) is not SoS, we deduce that
$
4\sigma^{-1}\big[m_{\mathcal{A}, \sigma_0}(\tilde{s}) - m^*_{(\sigma_0)}\big]
$
is also not SoS in \(s\). Consequently, \(m_{\mathcal{A}}(s) - m_{\mathcal{A}}(s^*) = m_{\mathcal{A}}(s) - 4\sigma^{-1}m^*_{(\sigma_0)}\) is not SoS for any $\sigma > 0$.
\end{proof}

\noindent
\Cref{lemma not SoS} can be verified numerically. 
Using an SoS relaxation (e.g., \texttt{SoSTOOLS}) for any numerically permissible $\sigma >0$, the moment-SDP solver returns a lower bound but not a global minimum, indicating that the SoS relaxation \emph{fails} to recover the true optimum\footnote{For very large $\sigma$ (e.g., $\sigma \ge 2500$),
the solver reports a global minimum close to the origin,
but this is due to numerical errors, not because the shifted polynomial becomes SoS.}.

\textbf{Connection to the Schnabel model.} 
For any $\sigma>0$, the polynomial \eqref{ex counterexample sch} can be written in the form~\eqref{sch} by choosing
$
a_j = e_j, b_j = 0, 
$
for $j = 1,2,3$ $g = 0$, and $H = \hat{H}$. Using  $s^* = 2\sigma^{-1/2}s^*_{(\sigma_0)}$ and 
\(s^*_{(\sigma_0)} = [-1.1498,\, 0.6674,\, 0.6674]^T\). We derive from \eqref{B_s star}, $B_{\mathcal{S}}(s^*)$ for $m_{\mathcal{A}}$ has the expression
\begin{eqnarray}
B_{\mathcal{S}}(s^*)
= H +  \diag ({s^*_{(\sigma_0)}}^2)
= \begin{bmatrix}
 9.2878 & 8 & 8 \\
8 & 5.7819 & 8 \\
8 & 8 &  5.7815
\end{bmatrix}
\label{bs example}
\end{eqnarray}
for any $\sigma>0$.  $B_{\mathcal{S}}(s^*)$  has eigenvalues
\([-2.2,  0,  23.1]\)  and is therefore indefinite for all $\sigma>0$. 
Thus, the sufficient SoS condition in \Cref{thm: SoS sch} fails.
Yet, we verify that $s^*$, the necessary global condition in \Cref{opt for sch}
is satisfied: 
$
B_{\mathcal{S}}(s^*)[e_j]^2 \ge 0, j=1,2,3.
$

\textbf{Replacing the separable quartic norm.} {By contrast, replacing the separable quartic term by the Euclidean quartic norm fully changes the picture.}
Let us now consider the Euclidean quartic regularization, 
\begin{equation}
\label{example quartic euclidean norm}
m_{\mathcal{E}}(s)
=
 2\|s\|^2 
+ 8(s_1s_2+s_1s_3+s_2s_3)+ 
\sigma \|s\|^4= \hat{H}[s]^2 + \sigma \|s\|^4
\qquad \sigma \ge 0.
\end{equation}
\(m_{\mathcal{E}}(s)\) is a nonconvex quadratic quartically regularized (QQR) polynomial.  
By \eqref{QQR Model}, we know that \(m_{\mathcal{E}}(s) - m_{\mathcal{E}}(s^*)\) is SoS for every \(\sigma \ge 0\).  
In \Cref{thm SoS for example}, we also derived the explicit expression for its global minimizers.

\begin{lemma}
\label{thm SoS for example}
{Let \(s^* \in \mathbb{R}^3\) satisfy 
$
s_1^* + s_2^* + s_3^* = 0,
\|s^*\| = \sigma^{-1/2}.
$
Then \(s^*\) is a global minimizer of \(m_{\mathcal{E}}\), and its minimum value is 
$m_{\mathcal{E}}(s^*) = -\sigma^{-1}.
$}
\end{lemma}

\begin{proof}
$
m_{\mathcal{E}}(s) + \sigma 
=
\sigma\big(\|s\|^2 - \sigma^{-1}\big)^2 
+ 4(s_1+s_2+s_3)^2,
$
which is nonnegative and SoS for all \(s\).
The lower bound \(-\sigma^{-1}\) is attained precisely when
\(s_1+s_2+s_3=0\) and \(\|s\|=\sigma^{-1/2}\).
\end{proof}

\noindent
{\Cref{thm SoS for example} shows that \(m_{\mathcal{E}}\) has infinitely many global minimizers: every point on the circle
$
\{ s \in \mathbb{R}^3 : s_1 + s_2 + s_3 = 0,\ \|s\| = \sigma^{-1/2} \}.
$}
Although \(m_{\mathcal{E}}\) is SoS, the moment relaxations numerical solvers 
(e.g., \texttt{yalmip}) return a rank-deficient
moment matrix. This is because 
\eqref{example quartic euclidean norm} has infinitely many
global minimizers—the rank condition for finite relaxation cannot be
satisfied. From \eqref{qqr opt}, we have
\begin{eqnarray}
B_{QQR} ({s^*}^2)=
H +  \sigma \diag ({s^*}^2)
=
\begin{bmatrix}
4+ \sigma \|s^*\|^2 & 8 & 8 \\
8 & 4+ \sigma \|s^*\|^2 & 8 \\
8 & 8 & 4+ \sigma  \|s^*\|^2
\end{bmatrix} = \begin{bmatrix}
8 & 8 & 8 \\
8 & 8 & 8 \\
8 & 8 &  8
\end{bmatrix}.
\label{bqqr example}
\end{eqnarray}
$B_{QQR} ({s^*}^2)$ has eigenvalue $0, 8, 24$ and $B_{QQR} ({s^*}^2) \succeq 0$ for all $\sigma >0$.  
The subtle difference between \eqref{bs example} and \eqref{bqqr example}
\(B_{\mathcal{S}}(s^*)\), the separable quartic norm does not guarantee
positive definiteness, whereas the Euclidean quartic norm does.  {The contrast between \eqref{bs example} and \eqref{bqqr example} shows the difference in the mechanism. Under separable quartic regularization, the regularization contributes only
coordinatewise curvature, which may still leave the associated matrix indefinite. Under
Euclidean quartic regularization, by contrast, the curvature is distributed uniformly
through the term $\sigma\|s^*\|^2I_n$, which is strong enough to enforce positive
semidefiniteness in this example.} 

\textbf{Robustness: the effect of perturbations on global minimizers and on SoS representability.} We introduce a small perturbation to
\eqref{example quartic euclidean norm}, either by perturbing a single coefficient,
\begin{equation*}
\hat{m}_{\mathcal{E}, 1}(s)
=
 \frac{\sigma}{4}\|s\|^4
+ 2\|s\|^2
+ 8\Big[(1+10^{-5})s_1s_2 + s_1s_3 + s_2s_3 \Big].
\end{equation*}
Alternatively, we can add small random perturbations to several entries
\begin{align*}
\hat{m}_{\mathcal{E}, 2}(s)
=
\sigma \|s\|^4
+ 2\|s\|^2 
 + 8\Big[(1+10^{-6}\,\xi_{12})s_1s_2
+ (1+10^{-6}\,\xi_{13})s_1s_3
+ (1+10^{-6}\,\xi_{23})s_2s_3 \Big],
\end{align*}
where each \(\xi_{ij}\sim \mathcal{N}(0,1)\) represents an independent Gaussian random variable. 
When we apply the moment-relaxation solver \texttt{yalmip} to both perturbed polynomials, it returns a successful status of \texttt{1}, indicating that the ring of infinite global minimizers breaks—after perturbation—into one or at most two global minimizers. Consequently, the rank condition is satisfied. 
Small perturbations in the data (i.e., in the coefficients) can change the numerical outcome returned by \texttt{yalmip}.  
This behaviour observed arises due to the lack of robustness of this specific example. More details can be found in Appendix \ref{appendix Global Minima for m perturb}. 

\textbf{Perturbing the cubic term.}
We also consider perturbing the cubic term \eqref{ex counterexample sch}, \[
m_{\mathcal{A},2}(s)
=
2\|s\|^2
+ 8(s_1 s_2 + s_1 s_3 + s_2 s_3)
+ 0.1(s_1^3 + s_2^3 + s_3^3)
+ \sigma (s_1^4 + s_2^4 + s_3^4)
\]
for $\sigma =4$. 
This polynomial attains its global minimum at
$
s^* = [-1.1722,\ 0.6700,\ 0.6700]^T, 
m_{\mathcal{A},2}(s^*) = -2.2409.
$
Using \texttt{SoSTOOLS}, we obtain a lower bound
\(\gamma^* = -2.3990\),
which is strictly less than the true global minimum, implying that
\(m_{\mathcal{A},2}(s) - \gamma^*\) is still not SoS.
In contrast, if we replace the separable quartic term by the Euclidean quartic term, we obtain the variant
\[
m_{\mathcal{E},2}(s)
=
2\|s\|^2
+ 8(s_1 s_2 + s_1 s_3 + s_2 s_3)
+ 0.1(s_1^3 + s_2^3 + s_3^3)
+ \sigma \|s\|^4. 
\]
Using \texttt{SoSTOOLS}, we verify that
\(m_{\mathcal{E},2}(s) - m_{\mathcal{E},2}(s^*)\) is SoS
for the tested values of \(\sigma \gtrsim 10\), not not SoS for all $\sigma > 0$. This observation also verifies \Cref{thm B psd} and the result for \eqref{QQR Model}.  
Without the cubic term, \(m_{\mathcal{E}}(s)\) has a QQR structure and is therefore SoS-exact with zero layer of relaxation for all \(\sigma > 0\). 
After adding a small cubic perturbation, the polynomial \(m_{\mathcal{E},2}(s)\) is no longer a quadratic–quartically regularized model but instead becomes a cubic–quartically regularized polynomial.  
In this case, the bound for \(\sigma\) follows the requirements of \Cref{thm B psd}: \(\sigma\) must be sufficiently large to ensure that \(B(s)\) is PSD and thus guarantee an SoS representation.  
Numerically, we observe that the threshold occurs around \(\sigma \gtrsim 10\).

{These examples show that the importance of the geometry of the quartic term: Euclidean quartic regularization can create the matrix positivity needed for SoS exactness, whereas separable quartic regularization may fail to do so even for arbitrarily large $\sigma$.}


To conclude the section, some structural classes---including the univariate case, the quadratic-quartic model,
and the special tensor model---admit SoS exactness for every $\sigma>0$, so the
large-regularization threshold is not intrinsic there. By contrast, the Schnabel model
shows that this behavior is not universal: replacing the Euclidean quartic regularization
by a separable quartic regularization can destroy the mechanism that makes the shifted
polynomial SoS. This reinforces the central message of the paper that the geometry of the
regularization is as important as its degree.

\section{Interpretation and Implications of the SoS Exactness Theory}
\label{sec examples}

{The purpose of this section is to interpret these results from a broader optimization perspective. We first explain why they are different and compare them to the classical NP-hardness result by Nesterov~\cite{nesterov2003random}. The key point is that these results arise in different settings. In this sense, our framework identifies a tractable regime in which nonnegative quartic polynomials admit exact SoS representations. We then discuss the implications of our theory for higher-order regularization methods, where quartically regularized cubic models arise naturally, and briefly illustrate this perspective on an example in sensor localization. }

\subsection{Connections to Nesterov NP-hardness Results}
\label{sec: Nesterov Example}

{
Nesterov~\cite{nesterov2003random} proved that classical NP-hardness results for cubic and quartic polynomial optimization ver the Euclidean ball. These results, however, do not contradict the present theory, because they arise in a different optimization setting. In particular,
they concern constrained formulations or Lagrangian reformulations in which the quartic
term does not play the same structural role as the Euclidean quartic regularization used here.}

\begin{theorem}[\cite{nesterov2003random, luo2010semidefinite}]
The following polynomial optimization problems are NP-hard.  
Let \(A_j \in \mathbb{R}^{n \times n}\) be positive semidefinite matrices.

\begin{enumerate} \setlength{\itemindent}{0pt}
    \item \textbf{Quartic maximization over the Euclidean ball}:
    \begin{equation}
        \max_{s \in \mathbb{R}^n} 
        \bigg\{
        f_{\mathcal N}(s)
        := \sum_{j=1}^{k} \left(A_j[s]^2\right)^2 
        \;:\;
        \|s\| = 1
     \bigg\}.
        \label{Nesterov np hard 1}
    \end{equation}

    \item \textbf{Cubic maximization over the Euclidean ball}:  
    where \(g_3(s)\) is a homogeneous cubic polynomial,
    \begin{equation}
        \max_{s \in \mathbb{R}^{n}} 
        \left\{
          g_3(s)
          \;:\;
          \|s\| = 1
        \right\}.
        \label{Nesterov np hard 2}
    \end{equation}

    \item \textbf{Quartic minimization under quadratic constraints}:  
    where \(g_4(s)\) is a homogeneous quartic polynomial,
    \begin{equation}
        \min_{s \in \mathbb{R}^n}
        \left\{
        g_4(s)
        \;:\;
        A_j[s]^2 \ge 1,\quad j = 1,\dots,k
        \right\}.
        \label{Nesterov np hard 3}
    \end{equation}
    Note that this formulation is an alternative representation of
    \eqref{Nesterov np hard 1}.
\end{enumerate}
\end{theorem}

We first study the quartic problem \eqref{Nesterov np hard 1}.  
Since \(f_{\mathcal N}\) is 4-homogeneous, maximizing over the unit ball or
over the sphere is equivalent:
$
\max_{\|s\| \le 1} f_{\mathcal N}(s)
=
\max_{\|s\| = 1} f_{\mathcal N}(s).
$
A Lagrangian reformulation is given by
\begin{equation}
\label{eq:nesterov-quartic LM}
\mathcal{L}_1(s, \lambda)
= \sum_{j=1}^{m} \left(A_j[s]^2\right)^2 + \lambda(\|s\|^2 - 1).
\end{equation}
The quartic polynomial \(\mathcal L(s,\lambda)\) is bounded below with respect to \(s\), but not with respect to $\lambda$. The Lagrangian formulation \eqref{eq:nesterov-quartic LM} is fundamentally different from our quartically regularized polynomial model in the following ways. {The key reason is that the Lagrangian formulations and the quartically regularized model
have fundamentally different geometry.}

\begin{remark}
Several key distinctions explain why quartic maximization over the Euclidean
ball is NP-hard, whereas sufficiently quartically regularized minimization is tractable. 
\begin{enumerate} \setlength{\itemindent}{0pt}
    \item \textbf{Unbounded in the Lagrange multiplier.}  
    The variable pair is \((s, \lambda)\).  
    While \(\mathcal L(s,\lambda)\) is quartic in \(s\), it contains \emph{no}
    regularization term in~\(\lambda\), hence
    \(\mathcal L(s,\lambda)\) is unbounded in~\(\lambda\).

    \item \textbf{No freedom to choose a regularization parameter.}  
    In quartic regularization, the coefficient \(\sigma > 0\) is a user-chosen
    stabilizing parameter.  
    In contrast, here \(\lambda^*\) is an \emph{optimization variable}
    determined by the KKT system, not a tunable regularization parameter.

    \item \textbf{Constraint active at global solutions.}  
    Any KKT point \((s^*,\lambda^*)\) corresponding to a global maximizer
    of \eqref{eq:nesterov-quartic LM} must satisfy the boundary constraint
    \(\|s^*\| = 1\), implying \(\lambda^* \neq 0\).  
    Thus, the constraint term
    \(\lambda(\|s\|^2 - 1)\) \emph{cannot} be ignored in the global analysis.

    \item \textbf{Degeneracy when $\|g\|=0$.}  
Consider the Lagrangian as a function of the joint variable $(s,\lambda)$.  
Apart from the quartic terms in $s$, $\mathcal{L}(s,\lambda)$ contains only a homogeneous cubic term and a constant term, and no first order coefficient $\|g\|=0$.  
As discussed in \Cref{remark what if g=0}, the case where the gradient vanishes is structurally special and requires careful, case-by-case analysis.
\end{enumerate}
    \label{remark why np hard}
\end{remark}

Similarly, for the cubic maximization over the Euclidean ball in \eqref{Nesterov np hard 2}, we can write the corresponding Lagrange multiplier formulation
\begin{equation}
\label{eq:nesterov-cubic LM}
\mathcal{L}_2(s, \lambda) = g_3(s) + \lambda(\|s\|^2 - 1).
\end{equation}
By writing the complementarity condition in quartic form, 
$\mathcal{L}(s,\lambda)$ becomes a quartic polynomial in $s$.  However, the four key differences highlighted in \Cref{remark why np hard} also apply here:  
$\mathcal{L}(s,\lambda)$ is unbounded in $\lambda$ (if $\lambda \rightarrow -\infty$ 
and $\|s\|<1$, then $\mathcal{L}(s,\lambda)\rightarrow -\infty$);  
there is no freedom to choose a regularization parameter (since 
$\lambda^*$ is an optimization variable, not a regularization weight);  
and the constraint is always active at a global maximizer because a cubic polynomial is unbounded; lastly,  besides regularization, we have a homogeneous cubic term with $\|g\| = 0$. 
However, if we modify \eqref{eq:nesterov-cubic LM} as
\begin{equation}
\label{cubic LM 2}
\min_{s \in \R^n} \mathcal{L}_{c,2}(s) = g_3(s) + \sigma(\|s\|^4 - 1),
\end{equation}
where $\sigma > 0$ is a regularization parameter that can be chosen a priori, 
then $\mathcal{L}_{c,1}(s)$ becomes a quartically regularized polynomial and is bounded below. 
{
Note that the modification is not equivalent to Nesterov's constrained problem.
Rather, it is intended only as an illustration of how introducing Euclidean quartic
regularization changes the structure of the polynomial.}

{Figure~\ref{fig change of SoS on sigma} illustrates a typical transition:
as $\sigma$ increases, the modified quartically regularized model becomes SoS-exact,
in agreement with the theory developed in Section~\ref{sec: main theory}.} We plot the behaviour of the global minimizer $\|s^*\|$ as a function of the regularization parameter $\sigma$. For $\sigma$ increasing from $0.1$ to $100$, we test whether 
$\mathcal{L}_c(s) - \mathcal{L}_c(s^*)$ is a sum-of-squares using 
\texttt{GloptiPoly}/\texttt{MSDP}.  
Non-SoS cases are marked with red crosses, while SoS cases  are plotted by blue circles with the corresponding norm $\|s^*\|$  against the regularization weight~$\sigma$.  
The right panel also includes the curve $O(\sigma^{-1/3})$ for comparison which is the theoretical bound of $\|s^*\|$   given in \Cref{thm bound for $s^*$}. 
We observe that $\mathcal{L}_{c, 2}(s) - \gamma$ becomes SoS exact with zero layer of relaxation as $\sigma$ increases. The verifies our theoretical result in \Cref{thm B psd} that sufficiently regularized nonnegative quartic polynomials are SoS. 

\begin{figure}[t]
    \centering
    \includegraphics[width=0.8\linewidth]{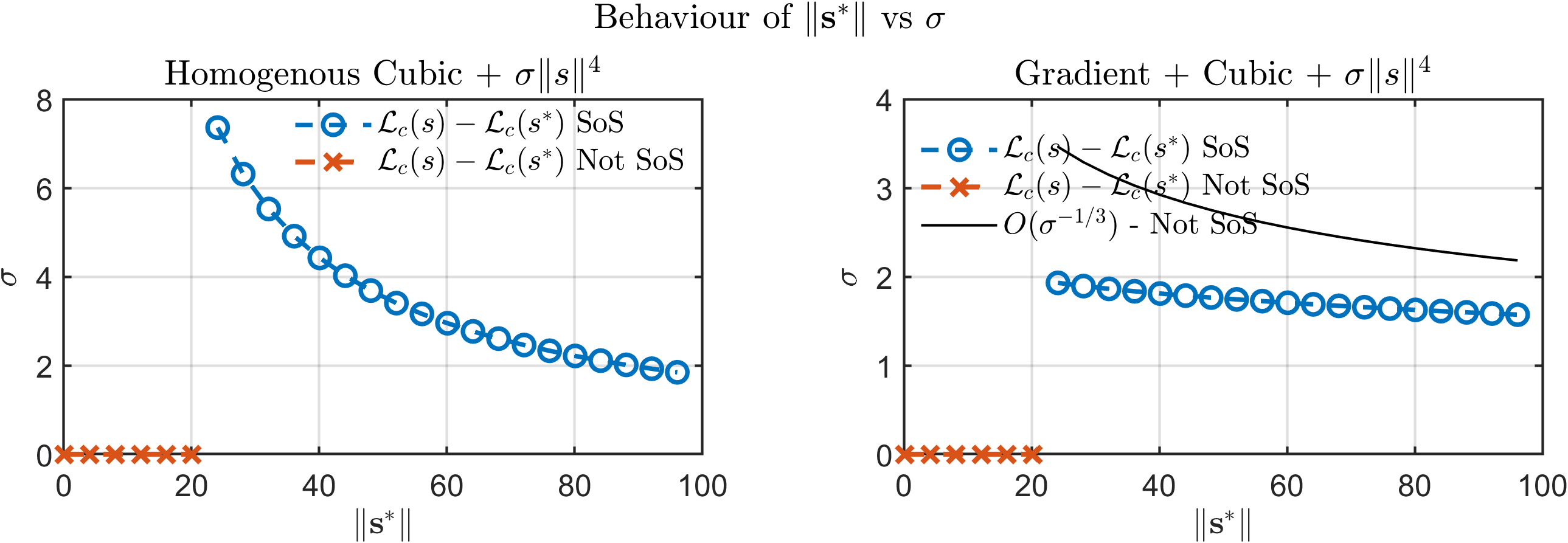}
\caption{\small
\textbf{Left:} Homogeneous cubic model 
$g_3(s) + \sigma\|s\|^4$ with random \emph{negative} coefficients 
uniformly generated in $[0,-100]$ and zero gradient term.  
\textbf{Right:} Non-homogeneous model 
$g(s) = g^Ts+ g_3(s)  + \sigma\|s\|^4$, where both the cubic and linear
coefficients are randomly generated in $[0,-100]$.  
}
\label{fig change of SoS on sigma}
\end{figure}

\begin{remark}
    Note that assume $A_j \succ 0$, a similar modification can be written for \eqref{eq:nesterov-quartic LM} with 
\begin{equation}
\label{cubic LM 1}
\min_{\s \in \R^n} \mathcal{L}_{1, c}(\s, \lambda)
=  \sigma \|\s\|^4 + (\s^T A_j^{-1} \s - 1).
\end{equation}
where we use the change of variables $\s^T \s = A_j[s]^2$. We observe same numerical results with $ \mathcal{L}_{1, c}$.
\end{remark}

It is important to emphasize that \eqref{cubic LM 2} is not equivalent to
\eqref{eq:nesterov-cubic LM}, and similarly \eqref{cubic LM 1} is not equivalent
to \eqref{eq:nesterov-quartic LM}. The former are unconstrained quartically
regularized models, whereas the latter are Lagrangian reformulations of constrained
optimization problems. The structural difference is essential.

{
{
Thus, the classical NP-hardness results and the present SoS-exactness theory should be
viewed as complementary rather than competing statements: they describe different polynomial
optimization geometries, and it is precisely the Euclidean quartic regularization that
moves the problem into a tractable regime.

{\subsection{Implications for Higher-Order Optimization Methods: An Illustrative Example}}

Quartic optimization problems arise in a range of applications, including independent component analysis~\cite{cardoso1998blind}, blind channel equalization~\cite{mariere2003blind}, and sensor network localization~\cite{biswas2006semidefinite}.  
We consider the sensor localization problem, where $\A$ and $\S$ denote the sets of anchor and sensor nodes, respectively. The goal is to estimate sensor positions $x_i \in \R^3$ by minimizing the quartic objective
\begin{equation}
\label{eq:sensor_localization}
    \min_{x_i \in \mathbb{R}^3,\, i\in \S} 
    f_{\S}(x)
    :=
        \sum_{i,j \in \S} 
            \big( \|x_i - x_j\|^2 - d_{ij}^2 \big)^2
        +
        \sum_{i \in \S,\, j\in\A}
            \big( \|x_i - s_j\|^2 - d_{ij}^2 \big)^2 ,
\end{equation}
where $d_{ij}$ are (possibly noisy) distance measurements and $s_j$ are known anchor locations. 
This problem is known to be NP-hard~\cite{luo2010semidefinite}. Moreover, even deciding whether the optimal value of~\eqref{eq:sensor_localization} is zero is NP-hard. In practical regimes with few anchors ($|\A| \le |\S|$), the objective is highly nonconvex and often admits multiple symmetric global minimizers. As a result, standard semidefinite relaxations may fail to satisfy the rank (flat extension) condition, preventing extraction of a global solution even when the relaxation is tight. 

From the perspective of higher-order methods, this problem is well-suited for third-order methods since the third-order tensor can be applied to directions without explicitly forming the full dense tensor; see \Cref{appendix: Derivative for Sensor–Localization model} for details. Motivated by our theoretical results, we consider an SoS-based AR3 framework in which the regularization parameter $\sigma$ is increased until the subproblem becomes SoS. The proposed \texttt{AR3 + SoS} framework directly minimizes the original third-order derivativeplus a quartic regularization term and preserves the full third-order derivative information. In this regime, the shifted model admits an exact semidefinite representation, and the global minimizer of the subproblem (if unique) can be recovered in polynomial time. The practical algorithmic framework and numerical setups are outlined in Appendix~\ref{appendix global arp}.

To illustrate this behavior, we consider a small but representative instance with $|\S| = |\A| = 2$ with a noise level varying from $10^{-5} -10^{0}$. This leads to a quartic problem in $n=6$ variables. Our numerical experiments indicate that the \texttt{AR3 + SoS} approach consistently achieves lower objective values than classical methods \footnote{More details on the methods are provided in \Cref{appendix global arp}}, such as \texttt{ARC}, \texttt{AR3 + ARC}, and \texttt{fminunc}, often reaching a local minimizer. It is worth noting that, we also tried to reformulate the $6$–variable quartic directly  \eqref{eq:sensor_localization} in SoS program and solve it with GloptiPoly/MSDP. Directly applying moment-SDP relaxations to \eqref{eq:sensor_localization} fails to recover a global minimizer due to the presence of multiple solutions and violation of the rank condition.

These observations highlight a key distinction: rather than solving the original quartic problem globally via a single large SDP, the proposed framework solves a sequence of structured, regularized subproblems for which SoS exactness can be enforced. This provides a potential pathway for leveraging semidefinite programming within higher-order optimization methods.

\begin{figure}[ht]
\centering
    \includegraphics[width=7cm]{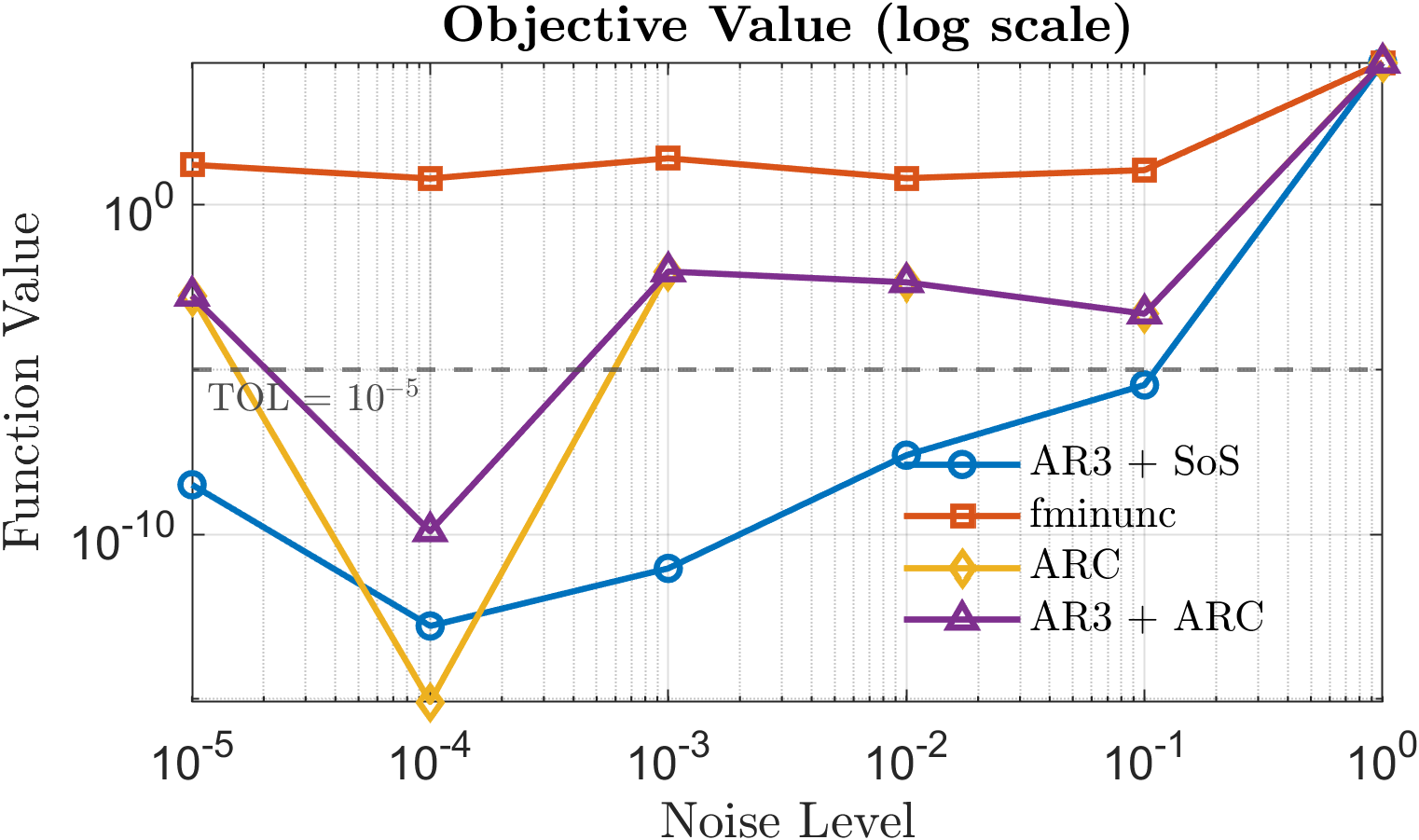}
\caption{\small Comparison of four methods on the sensor–localization problem with 
$2$ anchors and $2$ sensors ($n=6$ variables).  
Distances $d_{ij}$ are perturbed by adding Gaussian noise of magnitude $10^{-6}$ to $10^0$.  
All algorithms start from the anchor initialization.
 Function values at $x^*$ for different noise levels.} 
\label{fig:plot_table_side_by_side}
\end{figure}

\paragraph{Considerations.} 
We note that the global convergence of the \texttt{AR3 + SoS} algorithm, as well as the uniqueness of the subproblem minimizer, are separate questions and are left for future work. The current implementation relies on moment-SDP relaxations, which limits practical applicability to relatively small-scale problems due to the rapid growth of the SDP size $N = \binom{n+4}{4}$. 
Nevertheless, the structured nature of quartically regularized models—such as sparsity and low-rank tensor structure—suggests that significantly more scalable formulations may be possible. In particular, sparsity-adapted hierarchies and alternative structured SDP relaxations provide promising directions for reducing computational cost. Developing such scalable variants remains an important direction for future work.

\section{Conclusion and Outlook}
\label{sec: conclusion}

This paper studies when nonconvex quartically regularized polynomials admit exact sum-of-squares (SoS) certificates. While nonnegativity and SoS representations do not coincide in general and quartic polynomial optimization is NP-hard, our results show that Euclidean quartic regularization induces a fundamentally different regime.

The central insight is that regularization is not merely a regularizing tool, but a structural mechanism. 
Our main result establishes that, under mild assumptions and sufficiently large regularization, the shifted polynomial admits an exact SoS representation. This provides a constructive and explicit pathway from nonconvex polynomial models to tractable semidefinite programming formulations.  Beyond this asymptotic regime, we identify structured subclasses for which SoS exactness holds for all regularization parameters, including the quadratic--quartic model and a structured multivariate model with cubic term of the form $(t^Ts)\|s\|^2$. On the other hand, we show that this phenomenon is not universal: replacing the Euclidean quartic norm with separable quartic regularization can destroy SoS exactness even under arbitrarily large regularization. These results highlight that the geometry of the regularization term—rather than the polynomial degree alone—is the key factor governing SoS representability.

From an optimization perspective, our findings provide new insight into the role of quartic regularization in higher-order methods. Beyond ensuring descent and boundedness, it can place subproblems into a regime where global optimality becomes certifiable and, in certain cases, computationally tractable via semidefinite programming. Algorithmic framework and numerical experiments, including AR3 implementations and structured applications such as sensor localization, are left for future work. 

Several directions remain for future research. 
An important avenue is to further exploit the structure of quartically regularized models to design scalable semidefinite formulations. 
While standard SoS relaxations may suffer from the curse of dimensionality, the structured nature of the models studied in this paper—such as low-rank tensor representations and explicit optimality conditions—suggests that significantly smaller or more efficient formulations may be possible. Recent work provides encouraging evidence in this direction. For instance, tight semidefinite relaxations of dimension $O(n)$ have been derived for related cubic--quartically regularized models~\cite{zhou2025tight, zhu2023cubic}. 
Moreover, approaches based on sparsity and structured decompositions, including sparse SoS hierarchies and alternative relaxations such as DSOS/SDOS and BSOS~\cite{lasserre2017bounded,waki2006sums,weisser2018sparse,zheng2019sparse,ahmadi2019dSoS, li2025sos1}, offer promising direction for scalability. 
These developments suggest that, beyond worst-case complexity, structured quartically regularized polynomials may admit practically tractable semidefinite representations. Understanding how to systematically leverage such structure—through sparsity, low-rank tensor decompositions, or problem-specific reformulations—remains an important direction for future work. 
}}

\textbf{Acknowledgments}: This work was supported by the Hong Kong Innovation and Technology Commission
(InnoHK Project CIMDA). Coralia Cartis and Wenqi Zhu’s research was supported by the EPSRC grant
EP/Y028872/1, Mathematical Foundations of Intelligence: An “Erlangen Programme” for AI.

\appendix
\section{Proofs of the Main Results}

\subsection{\texorpdfstring{Proof for \Cref{lemma cubic form}}{Proof of lemma cubic}}
\label{appendix proof lemma cubic}

\begin{proof}
For a homogeneous cubic polynomial, $  c_3(v) = \frac{1}{6} T[v]^3$, we can write it in expanded form as 
$
c_3(v) := \sum_{1 \le i\le j \le k \le n}^n t_{ijk} v_iv_jv_k.  
$ 
Note that coefficients in the expanded form $\sum_{1 \le i\le j \le k \le n}^n t_{ijk} v_iv_jv_k$ correspond to the tensor entries  $\{T\}_{ijk}$ in tensor form $\frac{1}{6} T[v]^3$ as follows, $t_{ijk} = T_{ijk}$ for all distinct $ijk$, $t_{iij} =\frac{1}{2} T_{iij}$ and $t_{iii} = \frac{1}{6}T_{iii}$. 
We can further express  $c_3(v)$ as a sum of linear polynomials times quadratic polynomials
\begin{eqnarray}
    c_3(v):= \sum_{i,j, k = 1, i\le j \le k}^n  t_{ijk} v_iv_jv_k  &=& \bigg(\sum_{i=1}^n v_i^2 \bigg) \big( t^Tv\big) + \bigg[\sum_{1 \le i<j \le n} v_i v_j  \big({v^T b^{(ij)}}\big) \bigg] \notag
    \\ &=&  \bigg(\sum_{i=1}^n u_i \bigg) \big( t^Tv\big) + \bigg[\sum_{1 \le i<j \le n} z_{ij}  \big({v^T b^{(ij)}}\big) \bigg] 
    \label{Tv expression}
\end{eqnarray}
where $t = \{ t_{iii}\}_{1 \le i \le n} \in \R^n$. For $1\le i<j \le n$,  $b^{(ij)} = [t_{1ij}, \dotsc, t_{iij}^* \dotsc, t_{ijj}^*   \dotsc,t_{ijn}]^T \in \R^n$ with $t_{iij}^* = t_{iij} - t_{jjj} $ and $t_{ijj}^* : = t_{ijj} - t_{iii}$ . 
\end{proof}

\subsection{Proof of the SoS Representation Theorem}
\label{appendix proof thm 2.1}

In this appendix, we provide the full constructive proof of Theorem 2.1. 
The goal is to explicitly decompose the shifted polynomial
\[
q(v) = m_3(s^* + v) - m_3(s^*)
\]
into a sum of squares.

The proof proceeds in four main steps:
\begin{enumerate}
    \item We begin from the exact fourth-order expansion of $q(v)$ at $s^*$.
    \item We rewrite the cubic term $T[v]^3$ as a sum of quadratic monomials 
    multiplied by linear forms, introducing the variables $u_i = v_i^2$ and $z_{ij} = v_i v_j$.
    \item We complete squares for the quartic and mixed terms using a parameter 
    $\nu \in (0,\sigma)$, leading to three explicit sum-of-squares components.
    \item We collect the remaining terms into a quadratic form $B(s^*)[v]^2$ 
    and show that the entire expression is SoS whenever $B(s^*) \succeq 0$.
\end{enumerate}

We now give the detailed derivation.

\begin{proof}
Since  $\nabla m_3(s^*)=0$, from \Cref{remark q(v) expression}, 
\begin{eqnarray}
2 q(v) =   \nabla^2 m_3(s^*) [v]^2 + \frac{1}{3}T[v]^3 + \rr{2\sigma  \|v\|^2 v^T s^*}   +  \bb{\frac{\sigma}{2}\|v\|^4}.
\label{qv original for}
\end{eqnarray}
We aim to express every term in \eqref{qv original for} with $\omega  = [1,v, u, z]$. Fix a constant $\nu >0$, we write
\begin{eqnarray}
    \rr{\sigma v^T s^* \|v\|^2} &=&  (\sigma - \nu) v^T {s^*} \|v\|^2 + \nu \sum_{i=1}^n{s^*_iv_i^3} + \nu  \sum_{1 \le i < j \le n} v_i v_j (v_is^*_j+v_js^*_i) \notag
    \\ &=&  \rr{(\sigma - \nu)\bigg(\sum_{i=1}^n u_i \bigg)  (v^T {s^*}) + \nu \sum_{i=1}^n{s^*_iv_iu_i} + \nu  \sum_{1 \le i < j \le n} z_{ij} (v^T \tilde{s}^{(ij)}) }
\label{sigma vs expression}
\end{eqnarray}
where $\tilde{s}^{(ij)} = [0, \dotsc s^*_j, \dotsc, s^*_i, \dotsc 0] \in \R^n$ has $s^*_j$ on $i$th entry and $s^*_i$ on $j$th entry and all other zeros. Also,
\begin{eqnarray}
\bb{\sigma \|v\|^4 } = \sigma  \bigg(\sum_{i=1}^n v_i^2\bigg)^2 = \sigma   \bigg(\sum_{i=1}^n v_i^4 +  2 \sum_{1 \le i < j \le n} v_i^2 v_j^2   \bigg) = \bb{\sigma  \sum_{i=1}^n u_i^2 + 2 (\sigma - \nu)   \sum_{1 \le i < j \le n}  u_i  u_j  + 2  \nu \sum_{1 \le i < j \le n} z_{ij}^2. }
\label{sigma 4 expression}
\end{eqnarray}
Following from \eqref{Tv expression}, \eqref{qv original for}, \eqref{sigma vs expression},  \eqref{sigma 4 expression}, we have 
\small
\begin{eqnarray}
2q(v)  &=& \nabla^2 m_3(s^*) [v]^2  + \underbrace{\bb{\frac{\sigma - \nu}{2} \sum_{i=1}^n u_i^2  +  (\sigma - \nu) \sum_{1 \le i < j \le n} u_iu_j} + 
 2 \bigg(\sum_{i=1}^n u_i \bigg) \big( t^Tv\big) + \rr{2 (\sigma - \nu)\bigg(\sum_{i=1}^n u_i \bigg)  (v^T {s^*})}}_{:=\mathcal{I}_1} \notag
\\ && \qquad  + \underbrace{  \bb{ \nu  \sum_{1 \le i < j \le n} z_{ij}^2 } +  \sum_{1 \le i<j \le n} 2z_{ij}  \big({v^T b^{(ij)}}\big) + \rr{\nu  \sum_{1 \le i < j \le n} 2z_{ij} (v^T \tilde{s}^{(ij)})}}_{:=\mathcal{I}_2} + \underbrace{\rr{2 \nu \sum_{i=1}^n{s^*_iv_iu_i}}+  \bb{\frac{\nu}{2} \sum_{i=1}^n u_i^2}}_{:=\mathcal{I}_3}.  \label{2nd qv line}
\end{eqnarray}
\normalsize
Note that the expressions in \bb{blue} represents the homogeneous quartic terms  $\bb{\frac{\sigma}{2}\|v\|^4}$ from \eqref{sigma 4 expression}. The expressions in \rr{red} represents $\rr{\sigma v^T s^* \|v\|^2}$ from \eqref{sigma vs expression} and the rest of the cubic terms represents the tensor $\frac{1}{3}T[v]^3$ from \eqref{Tv expression}. 
$\mathcal{I}_1$, $\mathcal{I}_2$, and $\mathcal{I}_3$ can be reorganized as
\small
\begin{eqnarray*}
\mathcal{I}_1 &=& \underbrace{ \frac{\sigma - \nu}{2} \bigg[  \sum_{i=1}^n u_i + 2(v^Ts^*) + \frac{2t^Tv}{ \sigma -\nu} \bigg]^2}
_{:={\mathcal{SI}_1}}
-\underbrace{ \bigg( \frac{2}{\sigma - \nu} (t^Tv)^2 + 4 (v^Ts^*)  (t^Tv)\bigg)}_{\mathcal{BI}_1} - \pp{2 (\sigma - \nu)(v^Ts^*)^2} 
\\
\mathcal{I}_2 &=& \underbrace{\nu  \sum_{1 \le i < j \le n} \bigg[z_{ij} + v^T \tilde{s}^{(ij)}  +  \frac{v^T b^{(ij)}}{\nu} \bigg]^2}_{:={\mathcal{SI}_2}}
-  \underbrace{\sum_{1 \le i < j \le n} \bigg[ 2( v^T \tilde{s}^{(ij)} )
(v^T b^{(ij)})  +  \frac{(v^T b^{(ij)})^2}{\nu}}_{\mathcal{BI}_2} + \pp{\nu( v^T \tilde{s}^{(ij)} )^2}\bigg]
\\
\mathcal{I}_3 &=& \underbrace{\frac{\nu}{2}  \sum_{i=1}^n\big( u_i + 2s_i^*v_i\big)^2}_{:={\mathcal{SI}_3}} - \pp{2\nu \sum_{i=1}^n (s_i^*v_i)^2}.
\end{eqnarray*}
\normalsize
We collect the terms that contain multiple of $v_i$ and $s_i^*$, i.e., the last term in $\mathcal{I}_1$, $\mathcal{I}_2$, and $\mathcal{I}_3$, respectively (highlighted in \pp{purple}),
\begin{eqnarray}
  && 2 (\sigma - \nu)(v^Ts^*)^2 + \nu \sum_{1 \le i < j \le n} ( v_is_j^*+v_js_i^* )^2 + 2\nu \sum_{i=1}^n (s_i^*v_i)^2 \notag
 \\  &=&  2 (\sigma - \nu)(v^Ts^*)^2 + \nu  \sum_{1 \le i < j \le n} ( v_is_j^*+v_js_i^* )^2 + 2 \nu(v^Ts^*)^2 -  4 \nu \sum_{1 \le i < j \le n} s_i^*s_j^*v_iv_j \notag
 \\  &=& 2 \sigma (v^Ts^*)^2  + \underbrace{\nu  \sum_{1 \le i < j \le n}( v_is_j^*-v_js_i^* )^2 }_{:={\mathcal{BI}_3}}
 \label{2 vs}
\end{eqnarray}
Note that
\begin{eqnarray*}
    \mathcal{BI}_3 &=& \nu  \sum_{1 \le i < j \le n} v_i^2{s_j^*}^2+ v_j^2{s_i^*}^2 - 2 s_i^*s_j^*v_iv_j  = \nu \bigg(\sum_{i=1}^n v_i^2 \bigg) \|s^*\|^2  -\nu   \bigg(\sum_{i=1}^n s_i^*v_i \bigg)^2   = \nu v^T\bigg(\|s^*\|^2I_n - s^*{s^*}^T\bigg) v.
    \\  \mathcal{BI}_1 &=& \frac{2}{\sigma - \nu} (t^Tv)^2 + 4 (v^Ts^*)  (t^Tv) =   v^T\bigg[ \frac{2\, t t^T}{\sigma - \nu} +
       2 \big(s^* t^T + t {s^*}^T\big) \bigg]v.
  \\ \mathcal{BI}_2 &=&  \sum_{1 \le i < j \le n} \bigg[ 2( v^T \tilde{s}^{(ij)} ) 
(v^T b^{(ij)})  +  \frac{(v^T b^{(ij)})^2}{\nu}\bigg] =  v^T\bigg[
      \big(\tilde{s}^{(ij)}b^{(ij)} {}^T  + b^{(ij)} \tilde{s}^{(ij)} {}^T \big)
      + \frac{b^{(ij)} {b^{(ij)}}^T }{\nu}
    \bigg]v.
\end{eqnarray*}
Using $\mathcal{I}_1$, $\mathcal{I}_2$,  $\mathcal{I}_3$ , \eqref{2 vs} and $\mathcal{BI}_1$, $\mathcal{BI}_2$, and $\mathcal{BI}_3$ , 
\begin{eqnarray}
2q(v)   &=& \bigg(\nabla^2 m_3(s^*) [v]^2  - 2 \sigma (v^Ts^*)^2 - {\mathcal{BI}_1} -{\mathcal{BI}_2}-{\mathcal{BI}_3}\bigg) + {\mathcal{SI}_1}+{\mathcal{SI}_2}+{\mathcal{SI}_3} \notag
\\ &=& B(s^*) [v]^2 + {\mathcal{SI}_1}+{\mathcal{SI}_2}+{\mathcal{SI}_3} 
\label{sum of square expression in q}
\end{eqnarray}
where ${\mathcal{SI}_1}$ to ${\mathcal{SI}_3}$ are the nonnegative square polynomials in $v$ defined in  $\mathcal{I}_1$ to $\mathcal{I}_3$. Comparing our derivations with the expression of $B(s^*)$ in \eqref{eq:B-sstar}, we see that
the first three terms of $B(s^*)$ come from $\nabla^2 m_3(s^*) - 2 \sigma s^*{s^*}^T$, while the remaining terms in  $B(s^*)$ are from  ${\mathcal{BI}_1}$ to ${\mathcal{BI}_3}$.
The matrix \(B(s^*)\) depends only on \(s^*\) and the coefficients of the model \(m_3\), namely \(g\), \(H\), \(T\), and \(\sigma\).  
{Since $B(s^*)$ is symmetric and positive semidefinite, the quadratic form $B(s^*)[v]^2$ is a sum of squares. 
Together with the remaining sum-of-squares terms ${\mathcal{SI}_1}+{\mathcal{SI}_2}+{\mathcal{SI}_3}$, this implies that $q(v)$ is itself a sum-of-squares polynomial.}
Since $ q(v) := m_3(s^* + v) - m_3(s^*) \ge 0$ for all $v$,  we deduce that  \(s^*\) is a global minimizer.
\end{proof}

\begin{remark} \label{other form of Q n=2}
(Alternative form of $Q$ for $n=2$)
    Let  $\nu$, $\delta$, and $\beta$ be arbitrary constants representing the degrees of freedom in choosing the matrix $\Q$.
\small
$$
\Q = \begin{bmatrix} 
 0 &  0  &  0 & 0 & 0  & 0
\\0 &  \nabla^2 m_3(s^*) &  & \sigma s_1 + t_{111} & (\sigma - \nu)  s_1 + t_{122} + \beta &  \nu  s_2  + t_{112} + \delta
\\0 &   &  & (\sigma - \nu) s_2  - \delta &  \sigma s_2 + t_{222} &  \nu  s_1  - \beta
\\ 0 &  \ast  &  \ast &   \frac{\sigma}{2} &  \frac{\sigma}{2} - \frac{\nu}{2} & 0
\\ 0 &  \ast &  \ast   &  \frac{\sigma}{2} - \frac{\nu}{2}  &  \frac{\sigma}{2} & 0
\\ 0 &    \ast &   \ast &   0  &   0 & \nu
 \end{bmatrix}.
 $$
\normalsize
\end{remark}

\subsection{SoS Certificate implies Local Optimality}
\label{appendix proof sos local min}

\begin{lemma} (SoS certificate implies local optimality conditions)
    The conditions in \Cref{thm: certify SoS Quartically Regularized Polynomial} imply that both first- and second-order local optimality conditions of $m_3$ are satisfied at $s^*$; namely, $
\nabla m_3(s^*) = 0
$
and  
$
\nabla^2 m_3(s^*) \succeq 0$, respectively.
\end{lemma}

\begin{proof}
The condition  $
\nabla m_3(s^*) = 0
$ in \Cref{thm: certify SoS Quartically Regularized Polynomial} already gives  the first-local optimality conditions of $m_3$. It remains to show that $B(s^*) \succeq 0$ implies $
\nabla^2 m_3(s^*) \succeq 0$. Note that $
\nabla^2 m_3(s^*) = H + T[s^*] + \sigma\|s^*\|^2 I_n + 2\sigma\, s^*{s^*}^T $. We rearrange \eqref{eq:B-sstar} into
\begin{eqnarray}
 \nabla^2 m_3(s^*) &\succeq&  2\sigma\, s^*{s^*}^T +  \nu \big(\,\|s^*\|^2 I_n - s^* s^{*T}\big)  + \notag
\\ &&  \qquad  \underbrace{\frac{2}{\sigma-\nu}\, t t^T  + 2 (s^* t^T + t s^{*T})}_{:= \mathcal{BF}_1} +  \underbrace{\sum_{1 \le i < j \le n}     \left(
      \tilde{s}^{(ij)} b^{(ij)T}
      + b^{(ij)} \tilde{s}^{(ij)T}
      + \frac{1}{\nu} b^{(ij)} b^{(ij)T}   \right)}_{:= \mathcal{BF}_2}. 
      \label{second order dev m3}
\end{eqnarray}
We have
\begin{eqnarray}
\label{bf1}
\mathcal{BF}_1 =2 \bigg(\sqrt{(\sigma-\nu)^{-1}} t+ \sqrt{\sigma-\nu } s^*\bigg)\bigg(\sqrt{(\sigma-\nu)^{-1}} t+ \sqrt{\sigma-\nu}   s^*\bigg)^T-2 (\sigma-\nu ) s^* {s^*}^T
\end{eqnarray}
and
\begin{eqnarray}
\label{bf2}
\mathcal{BF}_2 = \sum_{1 \le i < j \le n}  \bigg(\sqrt{{\nu} }   \tilde{s}^{(ij)} + \sqrt{{\nu}^{-1}} b^{(ij)} \bigg) \bigg(\sqrt{{\nu} }   \tilde{s}^{(ij)} + \sqrt{{\nu}^{-1}} b^{(ij)} \bigg)^T - \nu    \sum_{1 \le i < j \le n} \tilde{s}^{(ij)} \tilde{s}^{(ij)T}.
\end{eqnarray}
Since  $\tilde{s}^{(ij)} = [0, \dotsc s^*_j, \dotsc, s^*_i, \dotsc 0] \in \R^n$, we can write  $
\tilde{s}^{(ij)} := s^*_j e_i + s^*_i e_j ,
$
where $\{e_k\}_{k=1}^n$ denotes the standard basis of $\mathbb{R}^n$. Then
$
\tilde{s}^{(ij)} \tilde{s}^{(ij)T}
= {s^*_j}^2\, e_i e_i^T + {s^*_i}^2\, e_j e_j^T
+ s^*_i s^*_j \big( e_i e_j^T + e_j e_i^T \big).
$
Summing over all $1 \le i < j \le n$, the resulting matrix has diagonal entries $ \sum_{\ell \neq k} s_\ell^2 = \|s^*\|^2 - s_k^2,
$ and off-diagonal entries $M_{pq} = s^*_p s^*_q \quad (p \neq q).$ The sum admits the compact representation
\begin{eqnarray}
\label{bf3}
\sum_{1 \le i < j \le n} \tilde{s}^{(ij)} \tilde{s}^{(ij)T}
= \|s^*\|^2 I + s^* {s^*}^T - 2\,\diag\{{s^*_1}^2,\dots,{s^*_n}^2\}.
\end{eqnarray}
Substituting \eqref{bf1}, \eqref{bf2}, \eqref{bf3} in \eqref{second order dev m3}, we obtain that
\begin{eqnarray*}
 \nabla^2 m_3(s^*) &\succeq&  \underbrace{2\sigma\, s^*{s^*}^T +  \nu \big(\,\|s^*\|^2 I_n - s^* s^{*T}\big)  -2 (\sigma-\nu ) s^* {s^*}^T  - \nu   \bigg(  \|s^*\|^2 I + s^* {s^*}^T - 2\,\diag\{{s^*_1}^2,\dots,{s^*_n}^2\} \bigg)}_{\mathcal{BF}} \notag
\\ &&  \qquad  2 \bigg(\sqrt{(\sigma-\nu)^{-1}} t+ \sqrt{\sigma-\nu } s^*\bigg)\bigg(\sqrt{(\sigma-\nu)^{-1}} t+ \sqrt{\sigma-\nu}   s^*\bigg)^T +
\\ &&  \qquad  \sum_{1 \le i < j \le n}  \bigg(\sqrt{{\nu} }   \tilde{s}^{(ij)} + \sqrt{{\nu}^{-1}} b^{(ij)} \bigg) \bigg(\sqrt{{\nu} }   \tilde{s}^{(ij)} + \sqrt{{\nu}^{-1}} b^{(ij)} \bigg)^T.
\end{eqnarray*}
Rearranging the first term gives
$
\mathcal{BF} = 2\nu \diag\!\big({s_1^*}^2,\dots,{s_n^*}^2\big),
$
which is positive semidefinite. The next two terms are also positive semidefinite. Hence,
$
\nabla^2 m_3(s^*) \succeq 0
$,  $B(s^*) \succeq 0$ implies the second-order local optimality conditions of $m_3$.  
\end{proof}

\subsection{\texorpdfstring{Proof for \Cref{Q corollary}}{Proof of Q corollary}}
\label{appendix proof Q corollary}

\begin{proof} 
Using  $\omega:=  [1,v, u, z]^T$ as given in \Cref{phi2}, we analyze $ \omega^T \Q \omega$ term by term. Using $\mathbf{1}^T u = \sum_{i=1}^n u_i$, the   $u$-$u$ block yields 
  $$
  u^T U u = \frac{\sigma}{2} \sum_{i=1}^n u_i^2 +  (\sigma - \nu) \sum_{1 \le i < j \le n} u_iu_j. 
  $$
 The   $z$-$z$ block gives $ \nu  \sum_{1 \le i < j \le n} z_{ij}^2 $. 
The $v$-$u$ block gives
\begin{eqnarray*}
    u^T D(s^*) v =  2 \bigg(\sum_{i=1}^n u_i \bigg) \big( t^Tv\big) + 2 (\sigma - \nu)\bigg(\sum_{i=1}^n u_i \bigg)  (v^T {s^*}) +2 \nu \sum_{i=1}^n{s^*_iv_iu_i}.
\end{eqnarray*}
The $v$-$z$ block gives
\begin{eqnarray*}
    u^T C(s^*) z =    \nu  \sum_{1 \le i < j \le n} z_{ij} (v^T \tilde{s}^{(ij)}) +  \sum_{1 \le i<j \le n} z_{ij}  \big({v^T b^{(ij)}}\big).
\end{eqnarray*}
By adding the terms from these blocks and comparing $\omega^T \Q \omega$ with \eqref{2nd qv line}, we obtain
$
\omega^T \Q \omega = 2 q(v).
$
If the conditions of \Cref{thm: certify SoS Quartically Regularized Polynomial} are satisfied, then $q(v)$ is a sum of squares. By \Cref{thm SoS alternative}, we conclude that $\Q \succeq 0$.
\end{proof}

\subsection{\texorpdfstring{Global Optimality Conditions for $m_3$}{Global Optimality Conditions for m3}}
\label{appendix global opt}

For a general $m_3$, global optimality conditions were established in \cite{zhu2025global}. Let $s^* \in \mathbb{R}^n$ satisfy $\nabla m_3(s^*) = 0$. 
If $s^*$ is a global minimum, then
\begin{equation}
    \tag{Necessary cond.}
B_N(s^*): = H + \frac{2}{3} T [s^*] +\sigma \|s^*\|^2 I_n  + \frac{\|T\|}{3} \|s^*\| I_n\succeq 0.
  \label{necessary tight}
\end{equation}
If the following is satisfied,
\begin{equation}
    \tag{Sufficient cond.}
B_S(s^*): = H +  \frac{2}{3}T [s^*]  +\sigma \|s^*\|^2 I_n  -\frac{\|T\| }{3} \|s^*\|I_n- \frac{\|T\| ^2}{18\sigma} I_n \succeq 0,
\label{sufficient tight}
\end{equation} 
then $s^*$ is a global minimum.

\begin{theorem}
\label{thm nece = suff}
Assume that $m_3$ is locally nonconvex at $s=0$, namely  $\nabla m_3(0) = H$ is not positive definite. 
Let \(s^* \in \mathbb{R}^n\) satisfy
$\nabla m_3(s^*)=0$. 
If \(\sigma\) is chosen such that
\begin{eqnarray}
    \sigma \ge 9 \max \bigg\{\frac{27 \lambda_*^3 }{\|g\|^2}, \quad \frac{ \|H\|^3 }{\|g\|^2} , \quad  \bigg( \frac{\|T\|^3}{ \|g\|}\bigg)^{1/2}  \bigg\}, 
    \label{sigma bound for necc suff}
\end{eqnarray}
where $\lambda_*: = \min \{-\lambda_{\min}(H), 0\} \ge 0$, $\|H\|$, $\|T\|$ are defined in \eqref{def of norm 1}--\eqref{def of norm 2}.   
Then, the following statements hold:
\begin{enumerate} \setlength{\itemindent}{0pt}
\item $\big(\frac{\|g\|}{3 \sigma} \big)^{1/3} \le \|s^*\| \le 2 \big(\frac{\|g\|}{ \sigma} \big)^{1/3}  $ and  $
\sigma \ge 3 \max \bigg\{
\lambda_* \|s^*\|^{-2}, 
\|T\| \, \|s^*\|^{-1}
\bigg\},$ 
    \item \(B_S(s^*) \succeq 0\), \(s^*\) is the global minimizer of \(m_3\), the necessary condition~\eqref{necessary tight} and sufficient condition~\eqref{sufficient tight} coincide.
\end{enumerate}
\end{theorem}

\begin{proof}
\textbf{To prove result 1:} 
If $ \sigma$ satisfies \eqref{sigma bound for necc suff}, then $ \sigma$ satisfies the condition in \Cref{thm bound for $s^*$}, and the bound on $\|s^*\|$ in the first result follows directly from \Cref{thm bound for $s^*$}.  
Moreover, from \eqref{sigma bound for necc suff}, we have
\begin{eqnarray*}
\sigma^{1/3} \ge  3^{5/3}\|g\|^{-2/3} \lambda_*   , \qquad 
 \sigma^{2/3} \ge   3^{4/3}   \|g\|^{-1/3} \|T\| . 
\end{eqnarray*}
We deduce that
\begin{eqnarray}
\sigma \ge 3  \max \bigg\{    3^{2/3} \sigma^{2/3} \|g\|^{-2/3} \lambda_*  ,  \quad 3^{1/3}  \sigma^{1/3}  \|g\|^{-1/3}   \|T\| \bigg\} \ge  3  \max \bigg\{\lambda_* \|s^*\|^{-2},  \quad    \|T\|  \|s^*\|^{-1}  \bigg\}
\label{sigma bound 0 SoS}
\end{eqnarray}
where the last inequality uses $\|s^*\| \ge \big(\frac{\|g\|}{3 \sigma} \big)^{1/3} $ (i.e.,  $  3^{1/3}{\sigma}^{1/3} \|g\|^{-1/3} \ge \|s^*\|^{-1}  $) from \Cref{thm bound for $s^*$}. 

\noindent \textbf{To prove result 2:} 
Given \eqref{sigma bound 0 SoS}, we have
\begin{eqnarray*}
   \sigma  > 6^{-1/2} \|T\| \|s^*\|^{-1} 
   \qquad &\Rightarrow& \qquad  \frac{\sigma}{3} \|s^*\|^2 >\frac{{\|T\|}^2} {18\sigma}, 
\\ \sigma  \ge 3 \|T\| \|s^*\|^{-1} 
\qquad &\Rightarrow& \qquad  
\frac{\sigma}{3}  \|s^*\|^2  \ge \|T\|\|s^*\| ,
\\\sigma  \ge 3 \lambda_* \|s^*\|^{-2}
\qquad &\Rightarrow& \qquad   
\frac{\sigma}{3} \|s^*\|^2  \ge  \lambda_*.
\end{eqnarray*}    
We deduce that
\[
B_S(s^*) \succeq 
\bigg(
\lambda_* - \|T\| \|s^*\| 
+ \sigma \|s^*\| 
- \frac{\|T\|^2}{18\sigma}
\bigg) I_n 
\succeq 0.
\]
Since $\nabla m_3(s^*)=0$ and \(B_S(s^*) \succeq 0\), it follows that \(s^*\) is the global minimizer of \(m_3\). 
Furthermore, as \(s^*\) is a global minimizer, the necessary condition \(B_N(s^*) \succeq 0\) also holds. 
By construction, we already have \(B_N(s^*) \succeq B_S(s^*) \succeq 0\), 
and thus the necessary condition~\eqref{necessary tight} 
and the sufficient condition~\eqref{sufficient tight} coincide.
\end{proof}


\section{Auxiliary Lemmas and Definitions}

\subsection{Definitions and Preliminaries for SoS Convex}
\label{def sos convex}

Let $\R[s]^{n\times n}$ be the real vector space of  $n \times n$ real polynomial matrices, where each entry of such a matrix is a polynomial with real coefficients.

\begin{definition} {(SoS-convex, \cite[Def. 2.4]{ahmadi2013complete} \cite{helton2010semidefinite})}  
$h(s)$ is SoS-convex if its Hessian $\nabla^2 h(s)$ is an SoS-matrix. A  symmetric polynomial matrix $\check{H}(s) \in \R[s]^{n\times n}$  is an SoS-matrix if there exists a polynomial matrix $\check{V}(s) \in \R[s]^{n_1\times n}$ for some $n_1 \in \mathbb{N}$, such that $\check{H}(s) = \check{V}(s)^T\check{V}(s)$. 
\label{def SoS convex}
\end{definition}

\noindent
Clearly, if \(h(s): \mathbb{R}^n \rightarrow \mathbb{R}\) is SoS-convex, then \(h(s)\) is convex~\cite{ahmadi2013complete, kojima2003sums}. 
Recently, Ahmadi et al.~\cite{ahmadi2023higher} proved that for a locally strongly convex quartically regularized polynomial ( \(H \succeq \delta I_n \succ 0\)), there exists a sufficiently large regularization parameter \(\sigma\) such that the polynomial becomes SoS-convex (see \Cref{SoS convex thm}). 

\begin{theorem}(\cite{ahmadi2023higher}, Lemma~3 and Thm~3)
If \(H \succeq \delta I_n \succ 0\), then the following sum of  squares programming
\begin{equation*}
   \bar{\sigma} := \min_{\sigma} \; \sigma
   \qquad \text{s.t.} \qquad \eqref{m3} \text{ is SoS-convex}
\end{equation*}
is feasible. Moreover, this program can
be reformulated as an SDP of size that is 
polynomial in $n$.
\label{SoS convex thm}
\end{theorem}

\noindent
By choosing $\sigma$ under this setup, the resulting polynomial is strongly convex, its global minimizer exists and is unique. 
Using the SoS-convexity, we deduce that the unique global minimizer can be obtained by solving an SDP whose size is polynomial in \(N = \binom{n+2d}{2d}\). 
Therefore, for locally convex regularized polynomials, both determining such a regularization parameter \(\sigma\) and globally minimizing the resulting SoS-convex polynomial reduce to solving semidefinite programs in polynomial time. 
The following lemma from \cite{helton2010semidefinite} connects nonnegative SoS-convex polynomials to SoS polynomials, and therefore ensures that a sufficiently regularized, nonnegative, locally strongly convex polynomial is SoS.

\begin{lemma} ({SoS Convex + Non-negativity $\Rightarrow$ SoS}, Lemma 8 in \cite{helton2010semidefinite}) {Assume  the global minimizer $s^*$ of $h(s)$ exists and $\nabla h(s^*) = 0$. Let $q(s): = h(s) -h(s^*)$. Then, clearly, $q(s): = h(s) -\nabla h(s^*)$ is a non-negative polynomial with $q(s^*) = 0$ and $\nabla q(s^*) = 0$. Moreover, if $q(s)$ is SoS convex (i.e., $\nabla^2 h(s)$ is an SoS-matrix), then $q(s)$ is SoS. }
    \label{lemma SoS convex}
\end{lemma}



\subsection{Matrix Inequalities and Norms}
\label{appendix tensor 2 norm and F norm}

\begin{eqnarray}
\|H\|:&=&  \max_{s \in \R^n, \|s\|=1} |H[s]^2| = \max_{i} \{|\lambda_i|\}, \qquad \lambda_*  := \max\{-\lambda_{\min}[H], 0 \}\label{def of norm 1}
\\
\|T\| :&=&\max_{s \in \R^n, \|s\|=1} |T[s]^3|, \qquad \|T\|_F :=  \sum_{i,j,k} T_{ijk}^2 , \qquad \|g\| = (\sum_{i=1}^n g_i^2 )^{1/2}.   \label{def of norm 2}
\end{eqnarray}
Clearly, $\|T\|_F \ge \|T\|$. To see this, for a symmetric third-order tensor \(T \in \mathbb{R}^{n\times n\times n}\), 
define its cubic (spectral) norm as
\[
\|T\| := \max_{\|s\|=1} |T[s]^3|
= \max_{\|s\|=1} |\langle T,\, s\otimes s\otimes s \rangle|.
\]
By the Cauchy--Schwarz inequality for tensor inner products, we have
\[
|T[s]^3| 
= |\langle T,\, s\otimes s\otimes s \rangle|
\le \|T\|_F \,\|s\otimes s\otimes s\|_F
= \|T\|_F \,\|s\|^3 
= \|T\|_F.
\]
Hence,
\[
\boxed{\|T\| \le \|T\|_F.}
\]
Moreover, using the mode-1 unfolding \(T_{(1)} \in \mathbb{R}^{n\times n^2}\), one obtains
\[
T[s]^3 = s^T T_{(1)} (s\otimes s)
\quad\Rightarrow\quad
|T[s]^3| \le \|T_{(1)}\|_2,
\]
and therefore
\[
\|T\| \le 
\min\{\|T_{(1)}\|_2,\, \|T_{(2)}\|_2,\, \|T_{(3)}\|_2\}
\le \|T\|_F.
\]
Equality \(\|T\| = \|T\|_F\) holds if and only if 
\(T\) is (up to a sign) a rank-one symmetric tensor of the form 
\(T = \alpha\, u\otimes u\otimes u\) with \(\|u\|=1\).  
A simple lower bound follows from the diagonal entries:
\[
\max_{1\le i\le n} |T_{iii}| \le \|T\|.
\]
Thus, \(\|T\|\) is always upper bounded by the Frobenius norm,
attains equality for rank-one symmetric tensors, 
and provides a tighter measure of the dominant cubic interaction captured by \(T\).

\section{Alternative Representations and Interpretations}

\subsection{Integral Representation of the Model}
\label{appendix integral representation}

In this section, we present an alternative representation of the shifted polynomial
\[
q(v)=m_3(s^*+v)-m_3(s^*)
\]
based on results from \cite{helton2010semidefinite}, and show its equivalence 
to the expansion used in the main text.

It is shown in Lemma 8 of \cite{helton2010semidefinite} that if $s^*$ satisfies 
$\nabla m_3(s^*)=0$, then
\begin{equation}
\label{eq double integral}
m_3(s^*+v) - m_3(s^*)
= v^T \left[\int_0^1 \int_0^t \nabla^2 m_3(s^* + rv)\,dr\,dt \right] v.
\end{equation}
Define
\[
\mathcal{I}(s^*, v)
:=
\int_0^1 \int_0^t \nabla^2 m_3(s^* + rv)\,dr\,dt.
\]
Then $s^*$ is a global minimizer if and only if
\[
\nabla m_3(s^*)=0
\quad\text{and}\quad
\mathcal{I}(s^*, v)\succeq 0
\quad\text{for all } v\in\mathbb{R}^n.
\]

We now show that this representation coincides with the expansion used in 
\Cref{remark q(v) expression}. Using
\[
\nabla^2 m_3(s) = H + T[s] + \sigma\big(\|s\|^2 I_n + 2ss^T\big),
\]
we compute
\begin{align}
\mathcal{I}(s^*, v)
&= \int_0^1 \int_0^t \Big[
H + T[s^*+rv]
+ \sigma\big(\|s^* + rv\|^2 I_n + 2 (s^* + rv)(s^* + rv)^T\big)
\Big] dr dt \notag
\\
&= \frac{1}{2}\nabla^2 m_3(s^*)
+ \frac{1}{6} T[v]
+ \frac{\sigma}{6}\Big(2 {s^*}^T v I_n + 4 s^* v^T \Big)
+ \frac{\sigma}{12} \Big(\|v\|^2 I_n + 2 vv^T\Big).
\label{other Q form}
\end{align}

Multiplying by $v^T$ and $v$, we obtain
\[
\mathcal{I}(s^*, v)[v]^2
=
\frac{1}{2}\nabla^2 m_3(s^*)[v]^2
+ \frac{1}{6} T[v]^3
+ \sigma {s^*}^T v \|v\|^2
+ \frac{\sigma}{4} \|v\|^4,
\]
which matches the expression in \eqref{expression m3 using m_*}.

\section{Additional Examples and Extensions}
\subsection{Proof for the univariate case}
\label{appendix proof for  univariate case}

\begin{proof}[Proof for \Cref{opt cond n=1}]
Let $q(v):=m_3(s^*+v) - m_3(s^*)$, using \eqref{expression m3 using m_*}, we have
\begin{eqnarray*}
2 q(v)&=& 2 \underbrace{\nabla m_3(s^*)^T}_{=0}  v+ 
\frac{1}{2} v^T \bigg[ \nabla^2 m_3(s^*)[v]^2 + \frac{1}{3} T[v]^3 + 2\sigma {s^*}^T  v  + \frac{\sigma}{2} \|v\|^2 \bigg]v
\\&\underset{\text{use }n=1}{=}& \frac{1}{2} v^T \bigg[ \bigg( H + \frac{1}{3}T[s^*] + \sigma \|s^*\|^2 \bigg) + 2 \sigma {s^*}^2  + \frac{2}{3}T[s^*]    + \frac{1}{3} T[v] + 2\sigma {s^*}  v  + \frac{\sigma}{2} v^2 \bigg]v
\\&=& \frac{1}{2} v^T \bigg[ H + \frac{1}{3}T[s^*] + \sigma \|s^*\|^2 - \frac{T^2}{18  \sigma}  \bigg]v + \frac{\sigma}{4} \bigg(v+2s^* - \frac{T}{3 \sigma} \bigg)^2 v^2.
\end{eqnarray*}

`$\Rightarrow$'  
Clearly, if \eqref{opt cond n=1} is satisfied, $m_3(s^*+v) - m_3(s^*) \ge 0$ for all $v$. Therefore, $s^*$ is the global minimum. 

`$\Leftarrow$' If $s^*$ is a global minimum, then it is a stationary point, therefore $\nabla m_3(s^*) =0$. 
Since $s^*$ is the global minimum $m_3(s^*+v) - m_3(s^*)  \ge 0$ for all $v \in \R^n$. By setting $v = -(2s^* + \frac{1}{3 \sigma} T)$, we derive that $ H + \frac{1}{3}T[s^*] + \sigma \|s^*\|^2  \ge \frac{T^2}{18  \sigma}$.
\end{proof}

\begin{proof}[Proof for \Cref{thm: SoS in n=1}]
Since $s^*$ is the global minimum, we have $\nabla m_3(s^*) =0$, according to \eqref{expression m3 using m_*}, we can write 
$$
 m_3(s^*+v)-m_3(s^*)  = \frac{1}{2}\begin{bmatrix}
  1 \\ v \\ v^2 
 \end{bmatrix} \underbrace{\begin{bmatrix} 0 & 0 & 0 \\
0 &  \nabla^2 m_3(s^*) & \sigma s^* + \frac{T}{6}\\ 0 & \sigma s^*  + \frac{T}{6} & \frac{\sigma}{2}
 \end{bmatrix}}_{Q_a}\begin{bmatrix}
  1 \\ v \\ v^2 
 \end{bmatrix}.
$$
Since $s^*$ is the global minimizer, by \Cref{lemma opt cond n=1}, \eqref{opt cond n=1} holds and
$$
\nabla^2 m_3(s^*) = H + T[s^*]+ 3\sigma \|s^*\|^2 \underset{\eqref{opt cond n=1}}{\ge} 2\sigma \|s^*\|^2 + \frac{2}{3}Ts^*+ \frac{T^2}{18  \sigma}.
$$
Note that the term in $ 3\sigma \|s^*\|^2$ is only true for $n=1$ case, in $n \ge 2$, the term will be $ \sigma (\|s^*\|^2 I_n + ss^T)$. Therefore, 
The determinant of the submatrix in $Q_a$, 
\begin{eqnarray*}
  \frac{\sigma}{2}\nabla^2 m_3(s^*) - \bigg(\sigma s^* +  \frac{T}{6}\bigg)^2 \ge  \sigma^2 \|s^*\|^2  + \frac{\sigma}{3}Ts^* + \frac{T^2}{36}- \bigg(\sigma s^*+  \frac{T}{6}\bigg)^2  =  0. 
\end{eqnarray*}
 We deduce that $Q_a \succeq 0$, therefore $ m_3(s^*+v)-m_3(s^*)$ is SoS. 
\end{proof}

\subsection{Proof for Optimality Conditions for Schnabel model}
\label{appendix Proof for Schnabel}

\begin{proof}[Proof for \Cref{thm: SoS sch}]
Let \( q_{\mathcal{S}}(v) := m_{\mathcal{S}}(s+v) - m_{\mathcal{S}}(s) \).   For any vector $s, v\in \R^n$, let $q_{\mathcal{S}}(v): =  m_{\mathcal{S}}(s+v)  -  m_{\mathcal{S}}(s)$. We have 
\begin{eqnarray}
q_{\mathcal{S}}(v): = \nabla  m_{\mathcal{S}}(s)^T v +\frac{1}{2} \nabla^2  m_{\mathcal{S}}(s)[v]^2  + \sum_{j=1}^k \bigg[\sigma_j  (a_j^Tv)^2 (s^Ta_j)
+ \frac{1}{6}  (a_j^Tv)^2(b_j^Tv) + \frac{\sigma_j}{4} (a_j^Tv)^4 \bigg].
\label{universal difference sch}
\end{eqnarray}
where $
\nabla^2  m_{\mathcal{S}}(s) = H + \sum_{j=1}^k (a_ja_j^T)(b_j^Ts) + \sum_{j=1}^k 3 \sigma_j  (a_j a_j^T) (a_j^Ts)^2 .
$ 
Using \eqref{universal difference sch} and \eqref{optimality condition sch}, we have
\begin{eqnarray}
2q_{\mathcal{S}}(v) &=&
 v^T \bigg[  B_{\mathcal{S}}(s^*) + 2\sum_{i=1}^k \sigma_j \bigg( (a_ja_j^T) s^*  + \frac{b_j}{6\sigma_j} \bigg)  \bigg( (a_ja_j^T) s^*  + \frac{b_j}{6\sigma_j} \bigg)^T \bigg]v \notag
  \\ && \qquad + 2\sum_{i=1}^k  \bigg[\sigma_j  (a_j^Tv)^2 s^T(a_j a_j^T)v  + 
+ \frac{1}{2}  (a_j^Tv)^2(b_j^Tv) + \frac{\sigma_j}{4} (a_j^Tv)^4 \bigg]
\notag
\\&=&  \underbrace{B_{\mathcal{S}}(s^*)}_{\succeq 0}[v]^2 +  \sum_{j=1}^k\frac{\sigma_j}{2} \bigg[ (a_j^Tv)^2 +2(a_j^T s^*) (a_j^T v) + \frac{b_j^Tv}{3\sigma_j}  \bigg]^2.
\label{SoS sch}
\end{eqnarray}
Note that the deduction of \eqref{SoS sch} is similar to the proof of \Cref{lemma opt cond n=1}.
Since \eqref{SoS sch} is symmetric and positive semidefinite, the quadratic form $B_{\mathcal{S}}[v]^2$ is a sum of squares. Together with the remaining sum-of-squares terms, this implies that $q_{\mathcal{S}}(v)$ is itself a sum-of-squares polynomial. To prove the second statement, $q_{\mathcal{S}}(v): =  m_{\mathcal{S}}(s^*+v) -  m_{\mathcal{S}}(s^*)$ is SoS and thus nonnegative for all $v$. Therefore, $s^*$ is the global minimum.
\end{proof}

\begin{proof}[Proof for \Cref{opt for sch}]
If a global minimum exists and $s^*$ is a global minimum, then it is a stationary point; therefore $\nabla m_{\mathcal{S}}(s^*) =0$. From \eqref{SoS sch}, 
\begin{eqnarray}
2q_{\mathcal{S}}(v) 
=  B_{\mathcal{S}}(s^*)[v]^2 + \sum_{j=1}^k  \frac{\sigma_j}{2} \left[ \bigg\|a_ja_j^T v + a_ja_j^T s^*  + \frac{b_j}{6\sigma_j} \bigg\|^2- \bigg\|a_ja_j^T s^* + \frac{b_j}{6\sigma_j}\bigg\|^2\right]^2.
\label{sch opt 2}
\end{eqnarray}
Note that the deduction of \eqref{sch opt 2} is similar to  is similar to the proof of \Cref{lemma opt cond n=1}. 

\begin{enumerate} \setlength{\itemindent}{0pt}
    \item 
If $v \in \text{Im}(A)^\perp$, then $a_ja_j^T v  = 0$ for all $j = 1, \dotsc, k$. Then, we have  $ B_{\mathcal{S}}(s^*) [v]^2 \ge 0$ for $v \in \text{Im}(A)^\perp$. 

    \item 
Otherwise, for each unit vector $v = \tilde{k}\frac{a_j}{\|a_j\|} \in \text{Im}(A)$, since $a_j^Ta_{\iota}= 0$ for $\iota \neq 0$. Therefore, let $s_c: =a_ja_j^T s^*  + \frac{b_j}{6\sigma_j}$, 
\begin{eqnarray}
2q_{\mathcal{S}}(v) =  \tilde{k}^2 \|a_j\|^{-2} B_{\mathcal{S}}(s^*)[ a_j ]^2 +    \frac{\sigma_j}{2} \left[ \big\|\tilde{k} a_j + s_c \big\|^2- \big\|s_c\big\|^2\right]^2.
\label{sch opt 3}
\end{eqnarray}
If $w_j(\tilde{k}):= \tilde{k} a_j +s_c$ and $s_c$ are  orthogonal, the \eqref{sch opt 3} becomes $2q_{\mathcal{S}}(v) = \tilde{k}^2 \|a_j\|^{-2} B_{\mathcal{S}}(s^*)[ a_j ]^2 \ge 0$ and we deduce that $B_{\mathcal{S}}(s^*)[ a_j ]^2 \ge 0$. On the other hand, if $w_j(\tilde{k}):= \tilde{k} a_j +s_c$ and $s_c$ are not orthogonal, the line $s_c + w_j(\tilde{k})$ intersects the ball centred at the origin of radius $\|s_c\|$  at two points, $s_c$ and $u_{j,*}$. Set $\tilde{k}:= k^*$ such that $\|\tilde{k} a_j +s_c \|^2= \|s_c\|^2.$  Since $s^*$ is the global minimum, we have $ m_{\mathcal{S}}(s^*+v) -  m_{\mathcal{S}}(s^*)  \ge 0$ for all $v \in \R^n$, we arrive at $ B_{\mathcal{S}}(s^*) [a_j]^2 \ge 0$. 
\end{enumerate}
\end{proof}

\subsection{Global Minimization and SoS Bounds: Additional Experiments}
\label{appendix Global Minima for m perturb}

\begin{table}[!h]
\centering
\caption{\small Comparison of Global Minima Found by MoM and \(\gamma^*\) Found by SoS}
\small
\resizebox{0.6\textwidth}{!}{\begin{tabular}{|l|c|c|}
\hline
\textbf{Objective} & \textbf{Global Mins Found by MoM} & \textbf{\(\gamma^*\) Found by SoS} \\ 
\hline
$m_{\mathcal{E}}$ & - & \(1 - 3.6655 \times 10^{-10}\) \\ 
\hline
$\hat{m}_{\mathcal{E}, 1}$ & 
\(\pm [0.7071 ,  -0.7071  , -0.0000
]^T\) & \(1 -4.0000 \times 10^{-5}\) \\ 
\hline
$\hat{m}_{\mathcal{E}, 2}$  - run 1 & 
\(\pm [0.7809  , -0.1830  , -0.5979]^T\) & \(1 -6.6080 \times 10^{-6}\) \\ 
\hline
$\hat{m}_{\mathcal{E}, 2}$  - run 2 & 
\(\pm[0.2089  ,  0.5790 ,  -0.7879]^T\) & \(1  -2.1461 \times  \ 10^{-5}\) \\ 
\hline
$\hat{m}_{\mathcal{E}, 2}$  - run 3 & 

\(\pm[0.7107  , -0.0072 ,  -0.7035]^T\) & \(1 -4.3092 \times 10^{-5}\) \\ 
\hline
$\hat{m}_{\mathcal{E}, 2}$  - run 4 &  
\(\pm[0.8096 ,  -0.3135   ,-0.4962
]^T\) & \(1 -5.6846 \times 10^{-5}\) \\ 
\hline
$\hat{m}_{\mathcal{E}, 2}$  - run 5 & 
\(\pm[0.7903   ,-0.2176  , -0.5728]^T\) & \(1 -2.8288 \times 10^{-5}\) \\ 
\hline
\end{tabular}
}
\label{table:comparison}
\end{table}

\subsection{Derivative and Tensor for Sensor–Localization model}
\label{appendix: Derivative for Sensor–Localization model}

Let $S$ be the set of sensor indices, $A$ the set of anchors, and let each
sensor/anchor position lie in $\R^3$.  For $(i,j)\in S\times S$ and
$(i,j)\in S\times A$, define the residuals
\[
r_{ij}(x) \;=\; \|x_i - x_j\|^2 - d_{ij}^2,
\qquad
\tilde r_{ij}(x) \;=\; \|x_i - s_j\|^2 - d_{ij}^2,
\]
so that the objective can be written as the quartic polynomial
\[
f(x)
= \sum_{i,j\in S} r_{ij}(x)^2
  + \sum_{i\in S,\, j\in A} \tilde r_{ij}(x)^2.
\]
Each term is a composition of a squared norm and a quadratic residual.
Consequently, the derivatives can be obtained by simple applications of the
chain rule.  For example, the gradient with respect to $x_i$ is
\[
\nabla_{x_i} f(x)
= 4 \sum_{j\in S}
    r_{ij}(x)\,(x_i - x_j)
  + 4 \sum_{j\in A}
    \tilde r_{ij}(x)\,(x_i - s_j),
\]
and the Hessian consists of rank–one and rank–two updates built from the
displacement vectors $x_i - x_j$ and $x_i - s_j$.  Differentiating once more,
the third derivative $\nabla^3 f(x)$ is a structured $3|\S|\times
3|\S|\times 3|\S|$ tensor whose nonzero blocks are linear combinations of
tensors of the form
$
(x_i - x_j)\otimes(x_i - x_j)\otimes(x_i - x_j),
$ and $
(x_i - s_j)\otimes(x_i - s_j)\otimes(x_i - s_j),
$
together with lower–rank symmetric permutations.  Thus, although the problem
is highly nonconvex, its third–order tensor has a highly
structured form (and potentially low rank depending on the location of the anchor). Also, the third–order tensor can be applied to directions without ever forming the full
dense tensor explicitly.  This makes the model particularly suitable as a
testbed for third–order methods.

\subsection{Adaptive Third Order Regularization Method with SoS Taylor Models}
\label{appendix global arp}

All experiments were performed on an AMD Ryzen 5 7600 6-Core Processor (3.80 GHz) with 16 GB of RAM. 
We test three methods: \texttt{AR3 + SoS}, \texttt{fminunc}, \texttt{ARC} (Adaptive Regularization with Cubics \cite{cartis2011adaptive}), and \texttt{AR3 + ARC} ( Third-order method with \texttt{ARC} as the inner solver \cite{cartis2020concise}).  

Unless otherwise stated, the parameters for all methods were set to 
$\rho_1 = 0.1$, $\rho_2 = 0.9$, $\eta = 0.1$ and $\gamma_1 = 0.5$, and $\gamma_2 = 2$, $\delta = 0.1$, $\sigma_0 = 100$. 
The inner solvers used for each method were \texttt{mcm} from \cite[Algorithm 6.1]{cartis2007adaptive} for \texttt{ARC},  
the \texttt{ARC} algorithm itself for \texttt{AR3 + ARC},  
and the \texttt{MSDP/MSOL} moment-SDP solver  from GloptiPoly~3 \cite{henrion2009gloptipoly} with \texttt{SEDUMI} for  \texttt{AR3 + SoS}.  
The outer-loop termination tolerance was set to $\|\nabla f\| < 10^{-3}$, and the maximum iteration count at 3000.
If an inner iteration was used, its tolerance was fixed at $10^{-6}$ and the maximum inner iteration count, at 3000. 

\begin{algorithm}[h]
\caption{Adaptive Third Order Regularization Method with SoS Taylor Models (AR3+SoS)}  
\textit{Initialization}: 
Choose $x_0 \in \R^n$, Set $\sigma_0, \sigma_{\min}>0$,  $\gamma_2 > 1 > \gamma_1 >0$ , $\rho_2 > \rho_1  >0$.  

\textit{Input}: $x_0$, $f_0 := f(x_0), g :=\nabla_x f(x_0), H :=\nabla_x^2 f(x_0), T := \nabla_x^3 f(x_0)$, $\sigma_k := \sigma_0$.  Tolerance $\epsilon_g$ s.t $0 < \epsilon_g \ll 1$, max iteration $k_{\max}$,  $k:=0$.

\textbf{Step 1: Test for termination.} If $\|\nabla_x f(x_k)\| \le \epsilon_g$ or $ k \ge k_{\max} $, terminate and output  $s_* := s_k$.

\textbf{Step 2: Step computation.}
Compute $s_k = \argmin_{s \in \R^n} m_3(s)$ by moment SDP solver where $m_3$ defined in \eqref{m3}  with coefficients,  $f_0 = f(x_k),$ $  g =\nabla_x f(x_k),$ $ H =\nabla_x^2 f(x_k),$ $ T= \nabla_x^3 f(x_k), $ $\sigma =  \sigma_k$. 

\eIf {Moment SDP solved successfully (\texttt{status==1})}
{Extract one global optimal solution, $s_k$. }
{Increase regularization: $\sigma_k$ by $\sigma_k:= \gamma_1 \sigma_k$ and repeat Step 2. }

\textbf{Step 3: Acceptance of trial point.} Compute $T_{3,x_k}(s_k)$ and define
$
\rho_k = \frac{f(x_k) - f(x_k+s_k)}{f(x_k) - T_{3,x_k}(s_k)}.
$

\textbf{Step 4: Regularization parameter update.} 

\If{$\rho_k > \rho_1$}{
    \emph{Successful Iter.} $x_{k+1} := x_k + s_k$;
}
\ElseIf{$\rho_k > \rho_2$}{
    \emph{Very Successful Iter.}
    Decrease regularization: $x_{k+1} := x_k + s_k$;
    $\sigma_{k+1}:= \max\{ \gamma_2 \sigma_k,\ \sigma_{\min} \}$;
}
\ElseIf{$\rho_k \le \rho_1$}{
    \emph{Unsuccessful Iter.} $x_{k+1} := x_k$;
    Increase regularization:
    $\sigma_{k+1}:= \gamma_2 \sigma_k$;
}

Repeat with $k:=k+1$. 

\label{practical AR3 + SoS algo}
\end{algorithm}

\scriptsize{
\bibliographystyle{plain}
\bibliography{sample.bib}
}

\end{document}